\documentclass[11pt,reqno]{amsart}
\usepackage[margin=1in]{geometry}
\usepackage{algorithm,algpseudocode}
\usepackage[fleqn,tbtags]{mathtools}
\usepackage[shortlabels]{enumitem}
\usepackage{verbatim}
\usepackage{tikz}
\usepackage{tikz-cd}
\usepackage{filecontents}
\usepackage{placeins}
\usepackage{mlmodern}
\usepackage{extarrows}
\usepackage{graphicx}
\usepackage{calc}
\usepackage[labelformat=simple]{subcaption}
\usepackage{adjustbox}
\usepackage{amsmath,amssymb}
\usepackage{eucal,bbm}
\usepackage[scr]{rsfso}

\usepackage{hyperref,xcolor}
\hypersetup{
    colorlinks,
    linkcolor={red!50!black},
    citecolor={blue!50!black},
    urlcolor={blue!80!black}
}

\DeclarePairedDelimiter\abs{\lvert}{\rvert}
\DeclarePairedDelimiter\norm{\lVert}{\rVert}

\usetikzlibrary{arrows.meta, positioning}
\newcommand\restr[2]{{
  \left.\kern-\nulldelimiterspace 
  #1 
  \vphantom{\big|} 
  \right|_{#2} 
  }}

\DeclareMathOperator{\diag}{diag}
\DeclareMathOperator*{\fl}{fl}
\DeclareMathOperator{\sign}{sign}
\DeclareMathOperator{\re}{Re}
\DeclareMathOperator{\im}{Im}
\DeclareMathOperator{\Wis}{\textsc{Wishart}}
\DeclareMathOperator{\WB}{\mathscr{WB}}
\DeclareMathOperator{\SB}{\mathscr{SB}}
\DeclareMathOperator{\E}{\mathscr{E}}

\theoremstyle{definition}
\newtheorem{theorem}{Theorem}[section]
\newtheorem{definition}[theorem]{Definition}
\newtheorem{lemma}[theorem]{Lemma}
\newtheorem{corollary}[theorem]{Corollary}
\newtheorem{proposition}[theorem]{Proposition}

\numberwithin{equation}{section}

\newcommand{\ur}{\mathsf{u}}
\newcommand{\tp}{{\scriptscriptstyle\mathsf{T}}}

\newcommand*{\method}[1]{\scriptscriptstyle\mathsf{#1}}
\newcommand*{\doublemethod}[2]{\scriptscriptstyle(\mathsf{#1},{#2})}
\newcommand*{\overeqU}[1]{\ensuremath{\mathrel{\overset{\method{#1}}{=}}}}
\newcommand*{\overeq}[1]{\mathrel{\overset{\method{#1}}{\resizebox{\widthof{\kern1.25pt\overeqU{\method{#1}}}}{\heightof{$=$}}{$=$}}}}
\newcommand*{\undereqU}[1]{\ensuremath{\mathrel{\underset{\method{#1}}{=}}}}
\newcommand*{\undereq}[1]{\mathrel{\underset{\method{#1}}{\resizebox{\widthof{\kern1.25pt\undereqU{\method{#1}}}}{\heightof{$=$}}{$=$}}}}
\newcommand*{\overundereq}[2]
{\mathrel{\overset{\method{#1}}{\mathrel{\underset{\method{#2}}{\resizebox{\widthof{\kern1.25pt\undereqU{\method{#2}}}}{\heightof{$=$}}{$=$}}}}}}

\newcommand*{\doubleovereq}[2]{\mathrel{\overset{\doublemethod{#1}{#2}}{\resizebox{\widthof{\kern1.25pt\overeqU{\doublemethod{#1}{#2}}}}{\heightof{$=$}}{$=$}}}}

\newcommand{\blue}[1]{{\color{blue}#1}}

\setlist[description]{font=\normalfont\itshape,labelindent=2em}

\allowdisplaybreaks

\begin{document}

\title{Summing divergent matrix series}

\author{Rongbiao Wang}
\address{Computational and Applied Mathematics, University of Chicago, Chicago, IL 60637}
\email{rbwang@uchicago.edu}

\author{JungHo Lee}
\address{Department of Statistics and Data Science, Carnegie Mellon University, Pittsburgh, PA  15213}
\email{junghol@andrew.cmu.edu}

\author{Lek-Heng Lim}
\address{Computational and Applied Mathematics Initiative, University of Chicago, Chicago, IL 60637}
\email{lekheng@uchicago.edu}

\begin{abstract}
We extend several celebrated methods in classical analysis for summing series of complex numbers to series of complex matrices. These include the summation methods of Abel, Borel, Ces\`aro, Euler, Lambert, N\"orlund, and Mittag-Leffler, which are frequently used to sum scalar series that are divergent in the conventional sense. One feature of our matrix extensions is that they are fully noncommutative generalizations of their scalar counterparts---not only is the scalar series replaced by a matrix series, positive weights are replaced by positive definite matrix weights, order on $\mathbb{R}$ replaced by Loewner order, exponential function replaced by matrix exponential function, etc.  We will establish the regularity of our matrix summation methods, i.e., when applied to a matrix series convergent in the conventional sense, we obtain the same value for the sum. Our second goal is to provide numerical algorithms that work in conjunction with these summation methods. We discuss how the block and mixed-block summation algorithms, the Kahan compensated summation algorithm, may be applied to matrix sums with similar roundoff error bounds. These summation methods and algorithms apply not only to power or Taylor series of matrices but to any general matrix series including matrix Fourier and Dirichlet series. We will demonstrate the utility of these summation methods: establishing a Fej\'{e}r's theorem and alleviating the Gibbs phenomenon for matrix Fourier series; extending the domains of matrix functions and accurately evaluating them; enhancing the matrix Pad\'e approximation and Schur--Parlett algorithms; and more.
\end{abstract}

\subjclass{15A16, 40D05, 40G10, 47A56, 65B10, 65F60}

\keywords{divergent matrix series, matrix series summability, matrix functions, Kahan summation, Pad\'e approximation, Schur--Parlett algorithm}

\maketitle

\section{Introduction}\label{sec:intro}

As we learned in calculus or real analysis, whenever we have an expression
\begin{equation}\label{eq:=}
\sum_{k=0}^\infty a_k=s
\end{equation}
for some  $a_k \in \mathbb{C}$, $k=0,1,2,\dots$, and $s \in \mathbb{C}$, the meaning of `$=$' is \emph{defined} to be the convergence of the sequence of partial sums $s_n \coloneqq \sum_{k=0}^n a_k$ to the limit $s$ in the standard Euclidean metric $\lvert \,\cdot\,\rvert$ on $\mathbb{C}$. In this case the series $\sum_{k=0}^\infty a_k$ is said to be convergent with value $s$; and if it does not meet this definition of convergence, then it is said to be divergent.

Because of its ubiquity and utility, we sometimes lose sight of the fact that such an interpretation of `$=$' in  \eqref{eq:=} is purely by convention, and not sacrosanct. A series divergent in the sense of the conventional definition may have a well-defined value under alternative definitions of `$=$' that are perfectly legitimate mathematically.  Take the harmonic series $\sum_{k=1}^\infty 1/k$ for illustration, well-known to be divergent in the conventional sense but as soon as we change, say, the choice of the metric from Euclidean to $p$-adic $\lvert \, \cdot \, \rvert_p$, it becomes convergent in the sense that $\lvert s_n -s\rvert_p \to 0$ for some value $s \in \mathbb{C}$ that depends on $p$ \cite{Boyd}. Indeed, a well-known result in $p$-adic analysis  \cite{robert} is that with a $p$-adic metric, a series $\sum_{k=0}^\infty a_k$ is convergent if and only if $\lim_{k \to \infty} a_k  = 0$, obviously false by the conventional definition of series convergence.

Even if we restrict ourselves to the Euclidean metric, which is what we will do in the rest of this article, the meaning of `$=$' still depends on a specific way to sum the values $a_k \in \mathbb{C}$, $k=0,1,2,\dots$. As was known to early analysts, there are many other reasonable ways to assign a value to a series that is divergent in the conventional sense, and such values are mathematically informative and useful in many ways \cite{hardy}.  As Hardy pointed out  \cite{hardy}, a \emph{summation method} just needs to be a function from the set of infinite series to values, assigning a sum to a series, which may or may not be convergent in the conventional sense. 

The first and best-known summation method is likely Ces\`aro summation \cite{cesaro1890multiplication}, that allows one to sum the Grandi series $1-1+1-1+\dotsb$ to $1/2$. The idea can be traced back even earlier to Leibniz, d'Alambert, Cauchy, and other predecessors of  Ces\`aro \cite{hardy, borel}.  Ces\`aro summation has the property of being \emph{regular}, i.e., for a series that is convergent in the conventional sense, the method gives an identical value for the sum. Regular summation methods have been studied extensively \cite{boos, hardy, peyerimhoff69, borel} and applied in various fields from analytic number theory \cite{wiener32} to quantum field theory \cite{glimmjaffe} to statistics \cite{guraukrajewski15}. Indeed summing divergent series is an important aspect of \emph{renormalization},  a cornerstone of modern physics  \cite{weinberg}, particularly in the renormalization technique of \emph{zeta function regularization} \cite{hawking}.

A main goal of our article is to show that many if not most of these summation methods for series of complex numbers extend readily and naturally to series of complex matrices. Take a toy example for illustration: the Neumann series
\begin{equation}\label{eq:neumann}
\sum_{k=0}^\infty X^k = (I-X)^{-1}
\end{equation}
if and only if the spectral radius of $X\in \mathbb{C}^{d\times d}$ is less than $1$. Again `$=$'  here is interpreted in the sense of conventional summation, i.e., the sequence of partial sums $S_n \coloneqq \sum_{k=0}^n X^k$ converges to $ (I-X)^{-1}$ with respect to any matrix norm $\lVert \, \cdot \, \rVert$. Let $\lambda(X)$ denote the spectrum of $X$ and $\mathbb{D} \coloneqq \{ z \in \mathbb{C} : \lvert z \rvert < 1 \}$ the  complex open unit disc. Depending on which method we use to sum the series on the left-hand side of \eqref{eq:neumann}, we obtain different interpretations of `$=$':
\begin{description}
\item[conventional] \eqref{eq:neumann} holds if and only if $\lambda(X) \subseteq \mathbb{D}$;
    \item[Abel] \eqref{eq:neumann} holds if and only if  $\lambda(X)\subseteq \overline{\mathbb{D}} \setminus \{1\}$;
    \item[Ces\`aro] \eqref{eq:neumann} holds if and only if $\lambda(X)\subseteq \overline{\mathbb{D}} \setminus \{1\}$ and the geometric and algebraic multiplicities are equal for each eigenvalue in $\lambda(X)\cap \partial \mathbb{D} \setminus \{1\}$;
    \item[Euler] \eqref{eq:neumann} holds if and only if $\lambda((I+P)^{-1}(P+X))\subseteq \mathbb{D}$ for some $P \succ 0$ commuting with $X$;
    \item[Borel] \eqref{eq:neumann} holds if and only if   $\lambda(X) \subseteq \{z \in \mathbb{C} :  \re(\lambda)<1\}$.
\end{description}
The last four summation methods will be defined in due course. In case the reader is wondering, although the matrix $(I - X)^{-1}$  is well-defined as long as $1 \notin \lambda(X)$, we will see that there is no natural method that will extend the validity of \eqref{eq:neumann} to all $X\in \mathbb{C}^{d \times d}$ with $\lambda(X)\subseteq \mathbb{C} \setminus\{1\}$.

In the toy example above, the series in question is a \emph{power series} where the $k$th term is a scalar multiple of $X^k$. The matrix summation methods in our article will apply more generally to \emph{any series} of matrices $\sum_{k=0}^\infty A_k$, where $A_k$ may not be Taylor in nature, i.e., $(X - \alpha I)^k$, but may be Fourier $\sin( kX)$, Dirichlet $ \exp( X \log k )$, Hadamard powers $X^{\circ k}$, or yet other forms not covered in this article, e.g., it could be defined by a recurrence relation $A_k = B A_{k-1} (I - A_{k-1})$ or randomly generated $A_k \sim \Wis(\Sigma, n)$.

So a second goal of our article is to provide practical numerical algorithms that complement our theoretical summation methods. These algorithms will allow us to compute, in standard floating point arithmetic, a matrix $\widehat{S} \in \mathbb{C}^{d \times d}$ that approximates the theoretical sum $S \in \mathbb{C}^{d \times d}$  of the series $\sum_{k=0}^\infty A_k$ given by the respective summation method.

These two aspects are complementary: There is no numerical method that would allow one to ascertain the convergence of a series, regardless of which summation method we use. A standard example is the harmonic series $\sum_{k=1}^\infty 1/k$; every numerical method would yield a finite value \cite[Section~4.2]{Higham_Accuracy}, which is completely meaningless since its true value is $+\infty$. On the other hand, most of the matrix series we encounter will have no alternate closed-form expressions, again regardless of which summation method we use. The only way to obtain an approximate value would  be through computing one in floating point arithmetic.  In summary, the theoretical summation method permits us to determine convergence; and its corresponding numerical algorithm permits us to determine an approximate value. We provide an overview of these two aspects of our work. 

\subsection*{Theoretical: regular summation methods}

As in the case of its scalar counterpart, a matrix summation method is a \emph{partial function}, i.e., possibly defined on a subset of its stated domain, from the set of $d\times d$ complex matrix series to a sum in $\mathbb{C}^{d\times d}$. We will generalize five classes of summation methods for scalar-valued series to matrix-valued ones. Figure~\ref{fig:tree} organizes them in a tree.
\begin{figure}[hbt]
    \centering
    \begin{tikzpicture}[on top/.style={preaction={draw=white,-,line width=#1}},
on top/.default=4pt, >={Classical TikZ Rightarrow[]}]
        \node (SM) {Summation Methods}[sibling distance = 15em, level distance = 10ex]
            child {node (S) {Sequential} [sibling distance = 6em, level distance = 9ex]
            child {node (N) {N\"{o}rlund} [sibling distance = 6em, level distance = 9ex]
            child {node (C) {Ces\`aro}}}
            child {node (E) {Euler}}}
            child {node (F) {Functional} [sibling distance = 6em, level distance = 9ex]
            child {node (AM) {Abelian means} [sibling distance = 6em, level distance = 9ex]
            child {node (A) {Abel}}}
            child {node (L) {Lambert}}
            child {node (M) {Mittag-Leffler}
            [sibling distance = 8em, level distance = 9ex] child {node (WB) {weak Borel}}
            child {node (SB) {strong Borel}}}};
        \draw [->>] (C) -- (A);
        \draw [->>] (WB) -- (SB);
        \draw [->>] (E) -- (WB.west);
        \draw [<-] (N) -- (C);
        \draw [<-, on top = 5pt] (AM) -- (A);
        \draw [<-] (M) -- (WB);
        \draw [<-] (M) -- (SB);
        \draw [<-] (SM) -- (S);
        \draw [<-] (SM) -- (F);
        \draw [<-] (S) -- (N);
        \draw [<-] (S) -- (E);
        \draw [<-] (F) -- (AM);
        \draw [<-] (F) -- (L);
        \draw [<-] (F) -- (M);
    \end{tikzpicture}
    \caption{Relations between various methods: $a \to b$ means $a$-summation is a special case of $b$-summation; $a \twoheadrightarrow b$ means $a$-summable implies $b$-summable.}
    \label{fig:tree}
\end{figure}
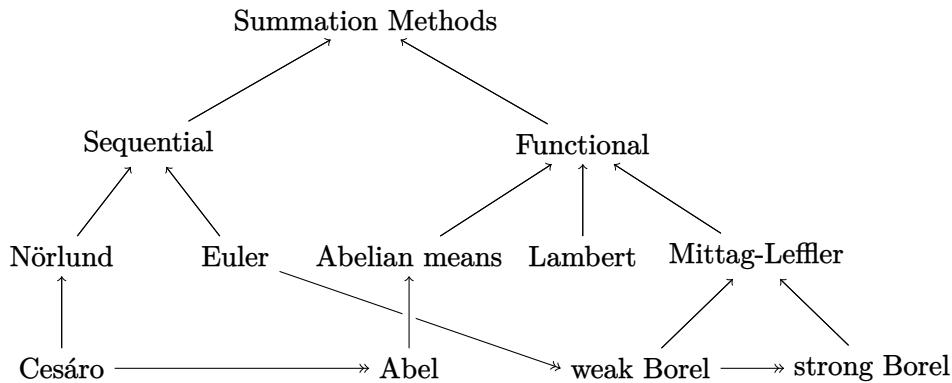

The five summation methods fall under two broad categories, \emph{sequential} and \emph{functional}  methods, discussed in Sections~\ref{sec:seq} and~\ref{sec:func} respectively. These terminologies follow those  for scalar-valued series \cite{borel}. Basically, a sequential method transforms the terms of a series or its sequence of partial sums into another sequence, whereas a functional method would transform them into a function. We will generalize two of the most important sequential methods, N\"{o}rlund (of which Ces\`aro is a special case) and Euler; and three of the most important functional methods, Lambert,  Abelian means, and Mittag-Leffler (Abel and Borel summations are respectively special cases of the latter two); showing that they also work for matrix series. 

One feature of our generalizations that we wish to highlight is that they are truly matrix-valued to the fullest extent possible. For example, our generalization of N\"{o}rlund summation  $\lim_{n\to \infty} (\rho_0+\dots + \rho_n)^{-1}(\rho_n s_0+\rho_{n-1}s_1+\dots+\rho_0s_n)$ with $s_n \coloneqq \sum_{k=0}^n a_k$ does not just replace $a_k \in \mathbb{C}$ by matrices $A_k \in \mathbb{C}^{d \times d}$ but also the positive scalars $\rho_0,\dots, \rho_n$ by positive definite matrices  $P_0,\dots,P_n$. Our extension of Abel summation $\lim_{x\to 0} \sum_{k=0}^\infty a_k e^{-\rho_k x}$ does not merely replace $a_k \in \mathbb{C}$ by matrices $A_k \in \mathbb{C}^{d \times d}$ but also the increasing sequence $0 < \rho_0 < \rho_1 < \cdots$ by  a sequence of matrices increasing in Loewner order  $0 \prec P_0 \prec P _1 \prec \cdots$ and $e^x$ by the matrix exponential function.

\subsection*{Practical: numerical summation algorithms}

Once a matrix series is ascertained to be summable via one of the aforementioned theoretical methods, the corresponding numerical method would be used to provide an approximate value in the form of a finite sum. However, it is nontrivial to obtain an accurate value for this finite sum in finite precision arithmetic. Simply adding terms in the finite sum in any fixed order would not give the most accurate result. In Section~\ref{sec:gen}, we adapt three numerical summation algorithms in \cite{HighamBlanchard,Higham_Accuracy} for sums of matrices: 
\begin{enumerate}[\normalfont(i)]
    \item \emph{block summation:} divide the finite sum into equally-sized blocks and sum the local blocks recursively, then sum the local sums recursively;
    \item \emph{compensated summation:} keep a running compensation term to extend the precision;
    \item \emph{mixed block summation:}  divide the finite sum into equally-sized blocks and sum the local blocks with one algorithm, then sum the local sums with another algorithm.
\end{enumerate}
We stress that the summation methods above apply to \emph{any} series of matrices, not just power series of matrices like those commonly found in the matrix functions literature \cite{funcofmat}.  In Section~\ref{sec:pow}, for the special case when we do have a matrix power series, we extend two algorithms for summing matrix power series  \cite{funcofmat}  by enhancing them with the summation methods introduced in Section~\ref{sec:seq}:
\begin{enumerate}[\normalfont(i),resume]
    \item \emph{Pad\'e approximation}: best rational approximation of a matrix function at a given order;
    \item \emph{Schur--Parlett algorithm}: Schur decomposition followed by block Parlett recurrence.
\end{enumerate}
We present numerical experiments in Section~\ref{sec:numer} to illustrate the value and practicality of our summation methods:
\begin{enumerate}[\normalfont(a)]
\item  using Ces\`aro summation to alleviate Gibbs phenomenon in matrix Fourier series;
\item  using Euler and strong Borel summations to extend matrix Taylor series;
\item  using Euler summations for high accuracy evaluation of matrix functions;
\item  using Ces\`aro and Euler summations in Pad\'e approximations;
\item  using Lambert summation to investigate matrix Dirichlet series;
\item  using compensated matrix summation for accurate evaluation of Hadamard power series.
\end{enumerate}
There are some surprises. For example, for \eqref{eq:neumann}, using Euler sum to evaluate the Neumann series and using \textsc{Matlab}'s \texttt{inv} to invert $I -X$, Euler sum gives results that are an order of magnitude more accurate than \textsc{Matlab}'s \texttt{inv}; Gibbs phenomenon in matrix Fourier series  happens only when the matrix involved is diagonalizable; a well-known property of the Riemann zeta function remains true for the matrix zeta function; etc. 

As one of our goals is to compute an approximate value in floating arithmetic for the matrix series and summation methods studied in this article, so even though some of our theoretical results readily extend to Banach algebras we do not pursue this unnecessary generality.

\section{Conventions and notations}\label{sec:notations}

Recall that it does not matter which matrix norm we use since all norms on a finite-dimensional space $\mathbb{C}^{d \times d}$ are equivalent and thus induce the same topology.  Throughout this article, we will use $ \norm{\,\cdot\,}$ to denote the Euclidean norm on $\mathbb{C}^d$ and spectral norm on $\mathbb{C}^{d \times d}$:
\[
    \norm{A} \coloneqq \sup_{x \ne 0} \frac{\norm{Ax}}{\norm{x}}.
\]
Note that the norm notation is consistent if we adopt the standard convention of identifying vectors in $\mathbb{C}^d$ with single-column matrices in $\mathbb{C}^{d \times 1}$. For $P \in \mathbb{C}^{d \times d}$, we use the shorthand $P \succ 0$ for $P$ positive definite, i.e., $x^* P x>0$ for all nonzero $x \in \mathbb{C}^d$. Recall that this condition implies that $P$ must also be Hermitian \cite[p.~80]{Zhang} (but the analogous statement is not true over $\mathbb{R}$). More generally $\succ$ denotes the Loewner order, i.e., $A \succ B$ if and only if $A-B$ is positive definite. We write $I$ for the identity matrix and $\mathbbm{1}$ for the all one's matrix. We use $\lambda(X)$ to denote the spectrum of $X \in \mathbb{C}^{d \times d}$. A note of caution is that we do not treat $\lambda(X)$ as a multiset; whenever we write $\lambda(X) = \{\lambda_1,\dots,\lambda_r\}$, the elements $\lambda_i$'s are necessarily distinct. For example, $\lambda(I) = \{1\}$ always, not $\{1,\dots,1\}$.

We write $\re(x)$ and $\im(x)$ for the real and imaginary parts of $x \in \mathbb{C}$, respectively. We write $\mathbb{D} \coloneqq \{ z \in \mathbb{C} : \lvert z \rvert < 1 \}$ for the open unit disc and $\mathbb{N} \coloneqq \{0,1,2,\dots\}$ for the nonnegative integers. Unless noted otherwise, all sequences, summands in a series, partial sums, will be indexed by $\mathbb{N}$ throughout this article. We denote closure and boundary of a set $\Omega$  by $\overline{\Omega}$ and $\partial \Omega$ respectively.

We also lay out some formal definitions and standard notations \cite{sequence,Dunford_Schwarz,partial,rudin} related to sequences and series of matrices.
\begin{definition}
    A \emph{matrix sequence} $A_{\bullet} \coloneqq (A_k)_{k=0}^\infty$ is a map from $\mathbb{N}$ to $\mathbb{C}^{d \times d}$ whose value at $k\in \mathbb{N}$ is denoted by $A_k \in \mathbb{C}^{d \times d}$. We say that the matrix sequence $A_{\bullet}$ converges to $A\in \mathbb{C}^{ d \times d }$ if
    \[
        \lim_{k\to \infty} \lVert A_k-A \rVert = 0,
    \]
    and denote it by $\lim_{k\to \infty}A_k = A$.  We denote the vector space of matrix sequences by
    \[
        s(\mathbb{C}^{d \times d}) \coloneqq \bigl\{ A_{\bullet} :  A_k \in \mathbb{C}^{d \times d}, k \in \mathbb{N} \bigr\}
    \]
and its subspace of convergent matrix sequences by
    \[
        c(\mathbb{C}^{d \times d}) \coloneqq \Bigl\{  A_{\bullet} : \lim_{k\to \infty}A_k = A  \in \mathbb{C}^{d \times d} \Bigr \}.
    \]
\end{definition}

We speak of a \emph{series} when we are interested in summing a sequence. Therefore, a series $\sum_{k=0}^\infty A_k$ and its underlying sequence $A_\bullet$ are one-and-the-same object and we will not distinguish them. While the convergence of a matrix sequence is unambiguous throughout this article, the summability of a matrix series is not and will take on multiple different meanings.  Getting ahead of ourselves, we will be defining the $\mathsf{R}$-sum $S$ of a series $\sum_{k=0}^\infty A_k$ and writing
\begin{equation}\label{eq:Rsum}
    \sum_{k=0}^\infty A_k \overeq{R} S
\end{equation}
where different letters in place of $\mathsf{R}$ would refer to N\"orlund means ($\mathsf{N}$), Ces\`aro ($\mathsf{C}$), Euler ($\mathsf{E}$), Abelian means ($\mathsf{A}$), Lambert ($\mathsf{L}$), weak Borel ($\mathsf{WB}$), strong Borel ($\mathsf{SB}$), and Mittag-Leffler ($\mathsf{M}$) summations, all of which will be defined in due course. We say that $A_\bullet$ is $\mathsf{R}$-summable if there is a well-defined $\mathsf{R}$-sum $S \in \mathbb{C}^{d \times d}$. The absence of a letter would denote conventional summation, i.e., $S$ is the limit of its sequence of partial sums, $S_\bullet = (S_k)_{k=0}^\infty$, $S_n \coloneqq \sum_{k=0}^n A_k$.

As usual, we write $C_b(\Omega) \coloneqq C_b(\Omega, \mathbb{C})$ for the Banach space of complex-valued continuous functions equipped with the uniform norm; $\Omega$ will usually be an open interval in $\mathbb{R}$. We will often have to discuss matrices whose entries are in $C_b(\Omega)$, i.e., $A(x) =[a_{ij}(x)]$ with continuous and bounded $a_{ij} : \Omega \to \mathbb{C}$, $i,j=1,\dots,d$. We denote the space of such matrices as  $C_b(\Omega)^{d \times d}$. These may also be viewed as matrix-valued continuous maps $A : \Omega \to \mathbb{C}^{d \times d}$ or as tensor product of the two Banach spaces \cite{acta}:
\[
C_b(\Omega)^{d \times d} = C_b(\Omega, \mathbb{C}^{d \times d}) = C_b(\Omega) \otimes \mathbb{C}^{d \times d}.
\]
Indeed the tensor product view will be the neatest as we will also need to speak of $C_b(\Omega)^{d \times d}$-valued sequences:
\[
s\bigl(C_b(\Omega)^{d \times d}\bigr) = s(\mathbb{C}^{d \times d}) \otimes C_b(\Omega), \qquad c\bigl(C_b(\Omega)^{d \times d}\bigr) = c(\mathbb{C}^{d \times d}) \otimes C_b(\Omega)
\]
but we will avoid tensor products for fear of alienating readers unfamiliar with the notion.

We will occasionally use the notion of a \emph{partial function} to refer to a map from a set $\mathcal{X}$ to a set $\mathcal{Y}$ defined on a subset $\mathcal{S}\subseteq \mathcal{X}$ called its \emph{natural domain}. These are useful when we wish to speak loosely of a map from $\mathcal{X}$ to $\mathcal{Y}$ that may not be defined on all of $\mathcal{X}$, but whose natural domain may be difficult to specify a priori. A matrix summation method  falls under this situation as we want to define a map $\mathsf{R}$ on $s(\mathbb{C}^{d \times d})$ that is only well-defined on its natural domain of $\mathsf{R}$-summable series.  Following convention in algebraic geometry, we write $\mathsf{R} : \mathcal{X} \dashrightarrow \mathcal{Y}$ to indicate that $\mathsf{R}$ may be a partial function.

\section{Sequential summation methods}\label{sec:seq}

Many summation methods for scalar series are \emph{sequential summation method} \cite{hardy}. In this section we will extend them to matrix series, defining various partial functions that map a sequence $A_\bullet \in s(\mathbb{C}^{d \times d})$ into a suitably transformed sequence in $c(\mathbb{C}^{d \times d})$ and defining the corresponding sum\footnote{Sometimes called \emph{antilimit} for easy distinction from the partial sums \cite{hardy}.} as the limit of the transformed sequence.  

Let $C_{n,k}\in \mathbb{C}^{d\times d}$, $n,k\in \mathbb{N}$, and consider the partial function $\mathsf{R} : s(\mathbb{C}^{d \times d}) \dashrightarrow c(\mathbb{C}^{d \times d})$ given by
\begin{equation}\label{eq:sequential-transformation}
    \mathsf{R}(A_\bullet)_n = \sum_{k=0}^\infty C_{n,k}S_k
\end{equation}
for any $A_\bullet \in s(\mathbb{C}^{d \times d})$, $S_n = \sum_{k=0}^n A_k$, and $n \in \mathbb{N}$. The $\mathsf{R}$-sum is defined to be $\lim_{n \to \infty } \mathsf{R}(A_\bullet)_n$, if the limit exists; and in which case we write \eqref{eq:Rsum} with $S =\lim_{n \to \infty } \mathsf{R}(A_\bullet)_n$.

For a conventionally summable series, we expect our summation method to yield the same sum, i.e., $\sum_{k=0}^\infty A_k \overeq{R} S$ whenever $\sum_{k=0}^\infty A_k = S$. This property is called \emph{regularity}. All summation methods considered in our article will be shown to be regular. In fact they satisfy the slightly stricter but completely natural condition of \emph{total regularity}: If a series sums to $+ \infty$ conventionally, then the method also sums it to $+\infty$. Total regularity is the reason why the validity of \eqref{eq:neumann} cannot be extended to all of $\mathbb{C} \setminus\{1\}$: For $d=1$, $\sum_{k=0}^\infty x^k = +\infty$ whenever $x \in (1,+\infty)$ and so no totally regular method could ever yield $(1 - x)^{-1}$ for all $x \in \mathbb{C} \setminus\{1\}$. We will not discuss total regularity in the rest of this article.

We provide a sufficient condition for the regularity of sequential summation methods  \eqref{eq:sequential-transformation}, generalizing \cite[Theorem~1]{hardy} to matrices.
\begin{theorem}\label{thm:LinearCharacterization}
Let $C_{n,k}\in \mathbb{C}^{d\times d}$, $n,k\in \mathbb{N}$. Suppose
\begin{enumerate}[\normalfont(i)]
    \item there exists $\eta>0$ such that $\sum_{k=0}^\infty \norm{C_{n,k}}<\eta$ for each $n \in \mathbb{N}$; \label{cond:1} 
    \item $\lim_{n\to \infty}C_{n,k}= 0$ for each $k\in \mathbb{N}$; \label{cond:2}
    \item $\lim_{n\to \infty}\sum_{k=0}^\infty C_{n,k} = I$. \label{cond:3}
\end{enumerate}
Then for any $A_\bullet \in s(\mathbb{C}^{d \times d})$ with  $S_n = \sum_{k=0}^n A_k$ and $\lim_{n \to \infty} S_n =S$, the series $\sum_{k=0}^\infty C_{n,k}S_k$ is summable in the conventional sense for each $n\in \mathbb{N}$ and 
\begin{equation}\label{eq:sequential-convergence}
    \lim_{n\to\infty} \sum_{k=0}^\infty C_{n,k}S_k = S.
\end{equation}
\end{theorem}
\begin{proof}
Since $S_\bullet = (S_k)_{k=0}^\infty$ is convergent and therefore bounded, $\norm{S_k} \le \beta$ for some  $\beta >0 $ and all $k\in \mathbb{N}$. It follows from \ref{cond:1} that for each $n \in \mathbb{N}$,
\[
    \sum_{k=0}^\infty \lVert C_{n,k}S_k \rVert \le \sum_{k=0}^\infty \lVert C_{n,k} \rVert \cdot \lVert S_k \rVert \le \beta \eta.
\]
Thus $\sum_{k=0}^\infty C_{n,k}S_k$ is summable. To show \eqref{eq:sequential-convergence}, first assume that $S=0$. For $\varepsilon>0$, choose $m\in \mathbb{N}$ sufficiently large so that $\norm{S_k} < \varepsilon/2 \eta$ for $k>m$. By \ref{cond:1} and \ref{cond:2},
\[
    \lim_{n\to \infty}\sum_{k=0}^m C_{n,k}S_k = 0 \qquad \text{and} \qquad \Bigl\|\sum_{k > m}C_{n,k}S_k \Bigr\| \le \frac{\varepsilon}{2 \eta}\sum_{k > m} \norm{C_{n,k}}\le \frac{\varepsilon}{2}. 
\]
Hence
\[
    \lim_{n\to\infty} \sum_{k=0}^\infty C_{n,k}S_k = \lim_{n\to \infty} \biggl( \sum_{k=0}^m C_{n,k}S_k +\sum_{k > m}C_{n,k}S_k \biggr) = 0.
\]
For $S\neq 0$, consider $A'_\bullet = (A'_k)_{k=0}^\infty$ with
\[
A_k' \coloneqq \begin{cases} A_0 -S & k = 0,\\ 
A_k & k=1,2, \dots,
\end{cases}
\]
with partial sums $S_k' = S_k-S$, $k \in \mathbb{N}$. Since $\lim_{k\to \infty} S_k' = 0$, we get $\lim_{n\to \infty}\sum_{k=0}^\infty C_{n,k}S_k' =0$.
By \ref{cond:3},
\[
    \lim_{n\to \infty}\sum_{k=0}^\infty C_{n,k} (S_k'+S) = \lim_{n\to \infty}\sum_{k=0}^\infty C_{n,k} S_k'+\lim_{n\to \infty}\sum_{k=0}^\infty C_{n,k}S = S.\qedhere
\]
\end{proof}
Theorem~\ref{thm:LinearCharacterization}  should be interpreted as follows: A summation method of the form \eqref{eq:sequential-transformation} satisfying \ref{cond:3} should be seen as taking matrix-weighted averages of the sequence of partial sums $S_\bullet$ for each $n$. Conditions \ref{cond:1} and \ref{cond:2} are what one would expect for weights: absolutely summable and not biased towards any partial sum $S_k$ respectively.  Theorem~\ref{thm:LinearCharacterization} would be useful for establishing regularity of matrix summation methods involving various choices of matrix weights.

\subsection{N\"{o}rlund means}\label{sec:norlund}

The scalar version of this summation method was first introduced in \cite{Woronoi} but named after N\"{o}rlund who rediscovered it \cite{norlund_original}. Its most notable use is for summing Fourier series \cite{Tamarkin,Shaheen_norlun}. Here we will extend it to series of matrices.
\begin{definition}
    Let $P_\bullet \coloneqq (P_k)_{k=0}^\infty$ be a sequence of positive definite matrices such that 
    \begin{equation}\label{eq:norlund}
        \lim_{k\to \infty} \lVert (P_0+\dots + P_k)^{-1} \rVert \lVert P_k \rVert = 0.
    \end{equation}
    For $A_\bullet = (A_k)_{k=0}^\infty \in s(\mathbb{C}^{d \times d})$, the series $\sum_{k=0}^\infty A_k$ is \emph{N\"{o}rlund summable} to $S\in \mathbb{C}^{d \times d}$ with respect to $P_\bullet$ if 
    \[
        \lim_{n\to \infty} (P_0+\dots + P_n)^{-1}(P_nS_0+P_{n-1}S_1+\dots+P_0S_n) = S,
    \]
    where $S_n = \sum_{k=0}^n A_k$, $n\in \mathbb{N}$. We denote this by
    \[
       \sum_{k=0}^\infty A_k \overeq{N} S
    \]
and call $S$ the N\"{o}rlund mean of $A_\bullet$. There is an implicit choice of $P_\bullet$ not reflected in the notation.
\end{definition}

\begin{corollary}[Regularity of N\"orlund mean]\label{cor:norlund}
    Let $P_\bullet \coloneqq (P_k)_{k=0}^\infty$ be a sequence of positive definite matrices satisfying \eqref{eq:norlund}. For $A_\bullet \in s(\mathbb{C}^{d \times d})$ and $S\in \mathbb{C}^{d \times d}$, if $\sum_{k=0}^\infty A_k =S$, then $\sum_{k=0}^\infty A_k \overeq{N} S$.
\end{corollary}
\begin{proof}
The N\"{o}rlund mean is a sequential summation method with a choice of
\[
    C_{n,k} = 
    \begin{dcases}
        (P_0+\dots + P_n)^{-1}P_{n-k} & \text{if } k\leq n,\\
        0 & \text{if } k>n,
    \end{dcases}
\]
in  \eqref{eq:sequential-transformation}.
To show regularity, we check the three conditions of Theorem~\ref{thm:LinearCharacterization}: Since
\[
    \sum_{k=0}^n (P_0+\dots + P_n)^{-1}P_{n-k} = I,
\]
the Conditions~\ref{cond:1} and \ref{cond:3} are satisfied. Since $P_0+\dots + P_n \succ P_0 + \dots + P_{n-k}$, we have $ \lVert (P_0+\dots + P_n)^{-1} \rVert \le \lVert (P_0+\dots + P_{n-k})^{-1} \rVert$ and so Condition~\ref{cond:2} is satisfied as
\[
    \lim_{n \to \infty} \lVert (P_0+\dots + P_n)^{-1} \rVert \lVert P_{n-k}\rVert  \le \lim_{n \to \infty} \lVert (P_0+\dots + P_{n-k})^{-1} \rVert \lVert P_{n-k}\rVert = 0 .\qedhere
\]
\end{proof}
The well-known  Ces\`aro summation is a special case of N\"{o}rlund summation \cite[Section~5.13]{hardy} with $P_\bullet$ given by
\[
   P_k = \binom{k+j-1}{j-1}I, \quad  k\in \mathbb{N},
\]
for $j\in\mathbb{N}\setminus \{0\}$. We extend the definition to matrices and write $(\mathsf{C},j)$ for the $j$th order Ces\`aro summation. In particular, $(\mathsf{C},1)$ is Ces\`aro summation extended to a matrix series, defined formally below.
\begin{definition}
    Let $A_\bullet = (A_k)_{k=0}^\infty \in s(\mathbb{C}^{d \times d})$ and $S_\bullet$ be its sequence of partial sums, $S_n = \sum_{k=0}^n A_k$. Define
    \[
        \Sigma_n \coloneqq \frac{1}{n}\sum_{k=0}^{n-1} S_k.
    \]
    The series $\sum_{k=0}^\infty A_k$ is \emph{Ces\`aro summable} to $S \in \mathbb{C}^{d \times d}$ if $\lim_{n\to\infty}\Sigma_n = S$. We denote this by 
    \[
        \sum_{k=0}^\infty A_k \overeq{C} S
    \]
and call $S$ the Ces\`aro sum of $A_\bullet$.
\end{definition}

A standard example of a  Ces\`aro summable (scalar-valued) series divergent in the usual sense is the Grandi series $1-1+1-1+\cdots$, which sums to $1/2$ in the  Ces\`aro sense. Indeed, this is a special case of $\sum_{k=0}^\infty x^k \overeq{C} 1/(1-x)$ for any $x \in \overline{\mathbb{D}} \setminus \{1\}$, which follows from
\begin{equation}\label{eq:scalar-cesaro}
    \lim_{n\to \infty}\frac{1}{n}\sum_{m=0}^{n-1}\sum_{k=0}^m x^k = \lim_{n\to \infty}\sum_{m=0}^{n-1}\frac{1-x^{m+1}}{n(1-x)} = \frac{1}{1-x}.
\end{equation}
We establish the corresponding result for matrix Neumann series, which  is more involved.
\begin{proposition}[Ces\`aro summability of Neumann series]\label{prop:cesaro}
Let $X\in \mathbb{C}^{d\times d}$. Then
\begin{equation}\label{eq:neumann_cesaro}
    \sum_{k=0}^\infty X^k \overeq{C} (I-X)^{-1}
\end{equation} 
if and only if $\lambda(X)\subseteq \overline{\mathbb{D}} \setminus \{1\}$ and the geometric and algebraic multiplicities are equal for each eigenvalue in $\lambda(X)\cap \partial\mathbb{D} \setminus \{1\}$.
\end{proposition}
\begin{proof}
We start with the backward implication. The assumption falls into two cases:
\begin{enumerate}[\normalfont(i)]
    \item $\lambda(X)\subseteq  \mathbb{D}$,
    \item $\lambda(X)\cap \partial\mathbb{D} \setminus \{1\} \neq \varnothing$ and the geometric and algebraic multiplicities are equal for each eigenvalue in $\lambda(X)\cap \partial\mathbb{D} \setminus \{1\}$.
\end{enumerate}
For the first case, we have $\sum_{k=0}^\infty X^k = (I-X)^{-1}$ and so \eqref{eq:neumann_cesaro} holds by Corollary~\ref{cor:norlund} since Ces\`aro summation is a special case of N\"orlund  summation. For the second case, let $J_1,\dots,J_r$ be a list of all Jordan blocks of $X$ corresponding to eigenvalues $\lambda_1, \dots,\lambda_j \in \partial\mathbb{D} \setminus \{1\}$  and $\lambda_{j+1},\dots \lambda_r \in \mathbb{D}$. By assumption, $\lambda_1, \dots,\lambda_j $ have equal geometric and algebraic multiplicities and thus the Jordan blocks $J_1,\dots,J_j$ are all $1 \times 1$. Hence the Jordan decomposition $X=WJW^{-1}$ has the form
\[
J =  \begin{bmatrix}
     \lambda_1\\
     &\ddots\\
     & &\lambda_j \\
     & & &J_{j+1}\\
     & & & &\ddots\\
     & & & & &J_r
\end{bmatrix},
\]
where $J_i \in \mathbb{C}^{d_i \times d_i}$, $i=j+1,\dots, r$. By \eqref{eq:scalar-cesaro}, 
\begin{alignat*}{3}
    \sum_{k=0}^\infty \lambda_i^k &\overeq{C} \frac{1}{1-\lambda_i} &\text{for } i&=1,\dots,j.
\intertext{By the first case,}
    \sum_{k=0}^\infty J_i^k &\overeq{C} (1-J_i)^{-1}  \quad &\text{for } i&=j+1,\dots,r.
\end{alignat*}
Therefore,
\[
    \sum_{k=0}^\infty \begin{bmatrix}
    \lambda_1^k\\
     &\ddots\\
     & &\lambda_j^k\\
     & & &J_{j+1}^k\\
     & & & &\ddots\\
     & & & & &J_r^k
\end{bmatrix}
\overeq{C} \begin{bmatrix}
    \frac{1}{1-\lambda_1}\\
     &\ddots\\
     & &\frac{1}{1-\lambda_j}\\
     & & &(I-J_{j+1})^{-1}\\
     & & & &\ddots\\
     & & & & &(I-J_r)^{-1}
\end{bmatrix},
\]
so
\[
    \sum_{k=0}^\infty X^k = \sum_{k=0}^\infty W J^k W^{-1} \overeq{C} (I-X)^{-1}.
\]

We next establish the forward implication. Suppose $X$ has spectral radius $\rho(X)>1$. As
\[
    \rho(X) = \lim_{k\to \infty} \lVert X^k \rVert^{\frac{1}{k}},
\]
for any $\epsilon < \rho(X)-1$, there is some $m\in \mathbb{N}$ such that $(\rho(X)-\epsilon)^k \le \lVert X^k \rVert$ for all $k>m$. In other words, $\lVert X^k \rVert$ grows exponentially and thus
\[
\lim_{n\to \infty}\lVert \Sigma_n \rVert = \lim_{n\to \infty} \Bigl \lVert \frac{1}{n}\sum_{m=0}^{n-1}\sum_{k=0}^{m}X^k \Bigr\rVert = \infty.
\]
Also, observe that 
\[
    \frac{1}{n}\sum_{m=0}^{n-1}\sum_{k=0}^m 1 = \frac{1}{n}\sum_{m=0}^{n-1}m = \frac{n-1}{2};
\]
so if $1$ is an eigenvalue of $X$, then its Neumann series cannot be Ces\`aro summable. Hence we must have $\lambda(X)\subseteq \overline{\mathbb{D}} \setminus \{1\}$. It remains to rule out the case where $X$ has a Jordan block of size greater than $1 \times 1$ for an eigenvalue in $\partial\mathbb{D} \setminus \{1\}$. Suppose $X$ has a  Jordan block $J_\lambda \in \mathbb{C}^{d_i \times d_i}$ with $d_i \ge 2$ and corresponding eigenvalue $\lambda \in \partial\mathbb{D} \setminus \{1\}$. Dropping the subscript $i$ to avoid clutter, we have
\begin{equation}\label{eq:Jordan}
    J_\lambda^k = 
    \begin{bmatrix*}[c]
        \lambda^k &\binom{k}{1}\lambda^{k-1}  &\binom{k}{2}\lambda^{k-2} &\ldots &\binom{k}{d-1} \lambda^{k-(d-1)}\\
         &\lambda^k &\binom{k}{1}\lambda^{k-1}   &\ldots &\binom{k}{d-2} \lambda^{k-(d-2)}\\
         & &\ddots &\ddots &\vdots\\
         & & &\lambda^k &\binom{k}{1}\lambda^{k-1}\\
         & & & &\lambda^k
    \end{bmatrix*} \qquad\text{for } k > d.
\end{equation}
Observe that the $(1,2)$th entry,
\[
    \biggl( \frac{1}{n} \sum_{m=0}^{n-1} \sum_{k=0}^m J_\lambda^k  \biggr)_{12} = \frac{1}{n}\sum_{m=0}^{n-1} \sum_{k=0}^m k\lambda^{k-1} = \sum_{k=0}^{n-1} \frac{(k-1)\lambda^k -k\lambda^{k-1}+1}{n(1-\lambda)^2}
\]
is divergent as $n\to \infty$. So the series $\sum_{k=0}^\infty J_\lambda^k$ is not Ces\`aro summable and neither is the Neumann series of $X$.
\end{proof}
The proof above shows that whenever there is a Jordan block of size greater than $1$ with eigenvalues on $\partial\mathbb{D} \setminus \{1\}$, Ces\`aro summation will fail to sum the Neumann series. In Section~\ref{sec:abelian}, we will see how we may  overcome this difficulty with Abel summation.

The best-known application of the scalar Ces\`aro summation is from Fourier Analysis \cite{Katznelson,tauberian_theory,young_1988}. It is well known that if $f \in L^2(-\pi,\pi)$, then its Fourier series
\[
  s_n(x) \coloneqq \sum_{k=-n}^n\widehat{f}(k) e^{ikx}, \qquad \widehat{f}(k) \coloneqq \frac{1}{2\pi}\int_{-\pi}^\pi f(x)e^{-ikx}\, dx,
\]
converges to $f$ in the $L^2$-norm, i.e., $\lim_{n\to \infty} \lVert s_n - f \rVert_2 = 0$. Fej\'{e}r's theorem \cite{fejer1903} gives the $L^\infty$-norm analogue for continuous functions with one caveat---the series has to be taken in the Ces\`aro sense: If $f \in C(-\pi, \pi)$, then
\[
    \sigma_n(x) \coloneqq  \frac{1}{n}\sum_{m=0}^{n-1}\sum_{k=-m}^m \widehat{f}(k) e^{ikx},
\]
converges uniformly to $f$, i.e., $\lim_{n\to \infty} \lVert \sigma_n - f \rVert_\infty = 0$.

A well-known consequence of Fej\'{e}r's theorem is that if $f\in L^2(-\pi,\pi)$ is continuous at $x\in (-\pi,\pi)$, then its Ces\`aro sum converges pointwise to $f(x)$ \cite{Katznelson,young_1988}. As an application of our notion of Ces\`aro summability for matrices, we extend Fej\'{e}r's theorem  to arbitrary matrices $X \in \mathbb{C}^{d \times d}$ with real eigenvalues and $2\pi$-periodic functions $f \in C^{d-1}(\mathbb{R})$. We emphasize that we do not require diagonalizability of $X$. While we have assumed that $f$ is $(d-1)$-times differentiable for simplicity, it will be evident from the proof that the result holds for any $f \in C(\mathbb{R})$ that is $(d_\lambda -1)$-times differentiable at each $\lambda \in \lambda(X)$ where $d_\lambda$ is the size of the largest Jordan block corresponding to $\lambda$. The exponential function, when applied to a matrix argument, refers to the matrix exponential \cite[Section~10.8]{funcofmat}. 
\begin{proposition}[Fej\'{e}r's theorem for matrix Fourier series]\label{prop:sylvester_Cesaro}
Let $X\in \mathbb{C}^{d\times d}$ have all eigenvalues real. Let $f \in C^{d-1}(\mathbb{R})$ be $2\pi$-periodic. Then
\begin{equation}\label{eq:sylvester_Cesaro}
        \lim_{n\to \infty}\frac{1}{n}\sum_{m=0}^{n-1}\sum_{k=-m}^m \widehat{f}(k) e^{ikX} = f(X).
\end{equation}
\end{proposition}
\begin{proof}
By the standard Fej\'er's theorem, for $j\leq d - 1$,
\[
    \lim_{n\to \infty}\sigma_n^{(j)}(\lambda) = \lim_{n\to \infty}\frac{1}{n}\sum_{m=0}^{n-1}\sum_{k=-m}^m (ik)^j\widehat{f}(k) e^{ik\lambda} = f^{(j)}(\lambda)
\]
for any $\lambda \in  \mathbb{R}$ and where the parenthetical superscripts denote $j$th derivative. For a Jordan block $J  \in \mathbb{C}^{d \times d}$ with eigenvalue $\lambda \in \mathbb{R}$,
\[
    f(J) = 
    \begin{bmatrix*}[c]
        f(\lambda) &f'(\lambda)  &\frac{f''(\lambda)}{2!} &\ldots &\frac{f^{(d-1)}(\lambda)}{(d-1)!}\\
         &f(\lambda) &f'(\lambda)   &\ldots &\frac{f^{(d-2)}(\lambda)}{(d-2)!}\\
         & &\ddots &\ddots &\vdots\\
         & & &f(\lambda) &f'(\lambda)\\
         & & & &f(\lambda)
    \end{bmatrix*}
\]
and so
\begin{align*}
    \lim_{n\to \infty}\frac{1}{n}\sum_{m=0}^{n-1}\sum_{k=-m}^m \widehat{f}(k) e^{ikJ} &= \lim_{n\to \infty}\frac{1}{n}\sum_{m=0}^{n-1}\sum_{k=-m}^m \widehat{f}(k)\begin{bmatrix*}[c]
        e^{ik\lambda} &ik e^{ik\lambda}  &\frac{(ik)^2 e^{ik\lambda}}{2!} &\ldots &\frac{(ik)^{d-1} e^{ik\lambda}}{(d-1)!}\\
         &e^{ik\lambda} &ik e^{ik\lambda}   &\ldots &\frac{(ik)^{d-2} e^{ik\lambda}}{(d-2)!}\\
         & &\ddots &\ddots &\vdots\\
         & & &e^{ik\lambda} &ik e^{ik\lambda}\\
         & & & &e^{ik\lambda}
    \end{bmatrix*}\\
    & =
    \begin{bmatrix*}[c]
        f(\lambda) &f'(\lambda)  &\frac{f''(\lambda)}{2!} &\ldots &\frac{f^{(d-1)}(\lambda)}{(d-1)!}\\
         &f(\lambda) &f'(\lambda)   &\ldots &\frac{f^{(d-2)}(\lambda)}{(d-2)!}\\
         & &\ddots &\ddots &\vdots\\
         & & &f(\lambda) &f'(\lambda)\\
         & & & &f(\lambda)
    \end{bmatrix*} =f(J).
\end{align*}
Now let $X = W\diag(J_1,\dots, J_r) W^{-1}$ be its Jordan decomposition with Jordan blocks $J_1,\dots,J_r$. Then
\begin{align*}
    \lim_{n\to \infty}\frac{1}{n}\sum_{m=0}^{n-1}\sum_{k=-m}^m \widehat{f}(k) e^{ikX} &= \lim_{n\to \infty}\frac{1}{n}\sum_{m=0}^{n-1}\sum_{k=-m}^m \widehat{f}(k) T\begin{bmatrix*}[c]
    e^{ikJ_1}\\ &\ddots\\ & &e^{ikJ_r}
\end{bmatrix*}T^{-1} \\
&= T\begin{bmatrix*}[c]
    f(J_1)\\ &\ddots\\ & &f(J_r) 
\end{bmatrix*}T^{-1}  = f(X). \qedhere
\end{align*}
\end{proof}
Proposition~\ref{prop:sylvester_Cesaro} provides a way to remedy the Gibbs phenomenon for matrix Fourier series, which we will illustrate numerically in Section~\ref{sec:gibbs}.

Before moving to our next method, we would like to point out that what may appear to be an innocuous change to a series could affect the value obtained using the summation methods in this article.  For example, if we had added zeros to every third term of the Grandi's series $1 - 1 + 1 - 1 + \dotsb$ to obtain the series $1-1+0+1-1+0+\dotsb$, its Ces\`aro sum decreases from $1/2$ to $1/3$.

\subsection{Euler method}\label{sec:euler}

Euler summation methods are another class of sequential summation methods. Its name comes from the $(\mathsf{E},1)$-method for scalar series, which involves the Euler transform \cite{tauberian_theory,hardy}. Here we will extend Euler transform and Euler  summation to matrices. Let $P \in \mathbb{C}^{d \times d}$ be a positive definite matrix. Emulating the calculation in \cite[Section~8.2]{hardy}, for $A_\bullet = (A_k)_{k=0}^\infty \in s(\mathbb{C}^{d \times d})$,
\[
\sum_{k=0}^\infty A_k = \sum_{k=0}^\infty \bigl((I+P)^{-1}[I-P(I+P)^{-1}]^{-1}\bigr)^{k+1} A_k =  \sum_{n=0}^\infty \sum_{k=0}^n \binom{n}{k} (I+P)^{-n-1} P^{n-k} A_k
\]
and thus we introduce the shorthand
\begin{equation}\label{eq:Aq}
    \E_n^P(A_\bullet) \coloneqq \sum_{k=0}^n \binom{n}{k} (I+P)^{-n-1} P^{n-k} A_k,
\end{equation}
and call it  the $P$-\emph{Euler transform} of $\sum_{k=0}^\infty A_k$.
\begin{definition}
    For $P \succ 0, A_\bullet \in s(\mathbb{C}^{d \times d})$, the matrix series $\sum_{k=0}^\infty A_k$ is \emph{Euler summable} to $S \in \mathbb{C}^{d \times d}$ with respect to $P$ or \emph{$(\mathsf{E},P)$-summable} to $S$ if 
    \[
        \sum_{n=0}^\infty \E_n^P(A_\bullet) =S.
    \] 
    We denote this by
    \[
        \sum_{k=0}^\infty A_k \doubleovereq{E}{P} S.
    \]
For the special case $P = \rho I$ where $\rho>0$ is a scalar, we just write $(\mathsf{E},\rho)$ instead of $(\mathsf{E},\rho I)$.
\end{definition}
Let $P \succ 0$. Then
\begin{align*}
    (I+P)^{n+1}\sum_{k=0}^n \E_k^P(A_\bullet) &= (I+P)^{n+1}\sum_{m=0}^n \E_m^P(A_\bullet) = \sum_{m=0}^n (I+P)^{n-m}\sum_{r=0}^m\binom{m}{r}P^{m-r}A_r \\
    &= \sum_{k=0}^n \sum_{r=0}^k \sum_{m=0}^n \binom{n-m}{k-r} \binom{m}{r} P^{n-k}A_r\\
    &= \sum_{k=0}^n \sum_{r=0}^k\binom{n+1}{k+1}P^{n-k} A_r = \sum_{k=0}^n \binom{n+1}{k+1} P^{n-k}S_k.
\end{align*}
Here we have used the Chu--Vandermonde's identity \cite{chu-vandermonde}: For any integers $0 \le r \le k \le n$,
\[
    \sum_{m=0}^n \binom{n-m}{k-r} \binom{m}{r} = \binom{n+1}{k+1}.
\]
The Euler method is thus a sequential summation method \eqref{eq:sequential-transformation} with a choice of 
\[
    C_{n,k} =
    \begin{dcases}
        \binom{n+1}{k+1} P^{n-k}(I+P)^{-n-1} & \text{if } k\leq n,\\[1ex]
        0 & \text{if } k>n.
    \end{dcases}
\]
\begin{corollary}[Regularity of Euler summation]\label{cor:regEu}
    For $A_\bullet \in s(\mathbb{C}^{d \times d})$ and $P,S\in \mathbb{C}^{d \times d}$ such that $P \succ 0$, if $\sum_{k=0}^\infty A_k = S$, then $\sum_{k=0}^\infty A_k \doubleovereq{E}{P} S$.
\end{corollary}
\begin{proof}
This follows directly from Theorem~\ref{thm:LinearCharacterization}, where the three conditions may be verified as follows: Since
\[
    \sum_{k=0}^\infty \binom{n+1}{k+1} P^{n-k}(1+P)^{-n-1} = I- P^{n+1}(I+P)^{-n-1} \prec I
\]
and $\lim_{n\to \infty} I- P^{n+1}(I+P)^{-n-1} = I$, Conditions~\ref{cond:1} and \ref{cond:3} hold. Condition~\ref{cond:2} follows from
\[
    \lim_{n \to \infty} \binom{n+1}{k+1} P^{n-k}(I+P)^{-n-1} = 0. \qedhere
\]
\end{proof}
Euler summability depends highly on the choice of $P \succ 0$. The next result partially characterizes it for commuting $P_1$ and $P_2$ via Loewner order. 
\begin{theorem}
    Let $A_\bullet \in s(\mathbb{C}^{d \times d})$ and $P_1, P_2 \in \mathbb{C}^{d \times d}$  be such that  $P_2 \succ P_1 \succ 0$ and $P_1P_2 = P_2P_1$. If $\sum_{k=0}^\infty A_k \doubleovereq{E}{P_1} S$, then $\sum_{k=0}^\infty A_k \doubleovereq{E}{P_2} S$.
\end{theorem}
\begin{proof}
For any $P \in \mathbb{C}^{d \times d}$ with $P \succ 0$ and $P_1P = PP_1$,
\begin{multline}\label{eq:transforms-eq}
\sum_{n=0}^m \sum_{k=0}^n \binom{m}{n}\binom{n}{k} P^{m-n} (I+P)^{-m-1} P_1^{n-k}(I+P_1)^{-n-1} A_k \\
= \sum_{k=0}^m \binom{m}{k} (P_1+ P + P_1 P)^{m-k}(I+P_1+ P + P_1 P)^{-m-1}A_k = \E^{P_1+ P + P_1 P}_m(A).
\end{multline}
Suppose $\sum_{k=0}^\infty A_k \doubleovereq{E}{P_1} S$. Since $\sum_{n=0}^\infty \E_n^{P_1}(A) = S$, it is Euler summable for any $P \succ 0$.
Set $P=(P_2-P_1)(I+P_1)^{-1}$. By the regularity in Corollary~\ref{cor:regEu}, $\sum_{n=0}^\infty \E_n^{P_1}(A) \doubleovereq{E}{P} S$, i.e.,
\[
  \sum_{m=0}^\infty \sum_{k=0}^m \binom{m}{k} P_2^{m-k}(I+P_2)^{-m-1}A_k = S.
\]
Hence $\sum_{k=0}^\infty A_k \doubleovereq{E}{P_2} S$.
\end{proof}
Equation~\eqref{eq:transforms-eq} reveals the composition rule for Euler transforms: If $P_1$ and $P_2$ commutes, then the $P_1$-Euler transform of the $P_2$-Euler transform is the $(P_1+P_2+P_1P_2)$-Euler transform. To gain more insights, we apply it to the Neumann series.
\begin{proposition}[Euler summability of Neumann series]
For $X\in \mathbb{C}^{d\times d}$, $P \succ 0$, and $PX = XP$,
\[
    \sum_{k=0}^\infty X^k \doubleovereq{E}{P} (I-X)^{-1}
\] 
if and only if  $\lambda\bigl((I+P)^{-1}(P+X)\bigr)\subseteq \mathbb{D}$.
\end{proposition}
\begin{proof}
By commutativity, the $P$-Euler transform of the Neumann series is
\begin{equation}\label{eq:eq_neumann}
    \E^P_n(X) =(I+P)^{-n-1\sum_{k=0}^n \binom{n}{k}P^{n-k}X^k} =(I+P)^{-n-1}(P+X)^n.
\end{equation}
Therefore,
\[
    \sum_{n=0}^\infty \E_n^P(X) = (I+P)^{-1} \sum_{n=0}^\infty ((I+P)^{-1}(P+X))^n,
\]
which is conventionally summable to $(I-X)^{-1}$ if and only if $\lambda\bigl((I+P)^{-1}(P+X)\bigr)\subseteq \mathbb{D}$.
\end{proof}
The case of $P = \rho I$ for a scalar $\rho>0$ is worth stating separately as they  commute with all $X\in \mathbb{C}^{d\times d}$.
\begin{corollary}\label{cor:euneu}
For $X\in \mathbb{C}^{d\times d}$, $\rho>0$,
\[
    \sum_{k=0}^\infty X^k \doubleovereq{E}{\rho} (I-X)^{-1}
\] 
if and only if  $\lambda(X)\subseteq \{ z \in \mathbb{C} : \lvert z + \rho \rvert < 1+\rho\}$.
\end{corollary}

Intuitively, choosing a ``small'' $P \in \mathbb{C}^{d \times d}$ ought to increase the rate of convergence. But it is difficult to obtain a universal relationship as the convergence rate invariably depends on the series. For scalar series, this is discussed in \cite{Knopp} and \cite{Rosser}. For matrix series, we will illustrate this numerically in Section ~\ref{sec:acc}.

Euler methods are generalized by the Borel methods. We will discuss their relationship in Section~\ref{sec:borel}. 

\section{Functional summation methods}\label{sec:func}

In sequential summation methods, we have a partial function $\mathsf{R} : s(\mathbb{C}^{d \times d}) \dashrightarrow c(\mathbb{C}^{d \times d})$ and the sum is the value of a limiting process as $n \to \infty$.  In functional summation methods, we have a partial function $\mathsf{R} : s(\mathbb{C}^{d \times d}) \dashrightarrow c(C_b(\Omega)^{d \times d})$ and the sum is the value of \emph{two} limiting processes---a sequential limit as $n \to \infty$ followed by a continuous limit as $x \to x_*$ in $\Omega$.

We now lay out the details. Let $\Omega \subseteq \mathbb{R}$ and $x_* \in \overline{\Omega}$ or $x_* =\infty$. Let $\mathsf{R} : s(\mathbb{C}^{d \times d}) \dashrightarrow c(C_b(\Omega)^{d \times d})$ be a partial function defined by
\begin{equation}\label{eq:functional-transformation}
    \mathsf{R}(A_\bullet)(x) = \sum_{k=0}^\infty B_k(x)A_k
\end{equation}
for some $B_k \in C_b(\Omega)^{d \times d}$, $k \in \mathbb{N}$. The $\mathsf{R}$-sum is defined to be
\[
S \coloneqq \lim_{x\to x_*} \mathsf{R}(A_\bullet)(x)
\]
if the limit exists, and in which case we write $\sum_{k=0}^\infty A_k\overeq{R} S$. The careful reader might notice that while we wrote   $\mathsf{R} : s(\mathbb{C}^{d \times d}) \dashrightarrow c(C_b(\Omega)^{d \times d})$, \eqref{eq:functional-transformation} seems to imply that $\mathsf{R} : s(\mathbb{C}^{d \times d}) \dashrightarrow C_b(\Omega)^{d \times d}$. The reason is that for a convergent series we do not distinguish between its sequence of partial sums in  $c(C_b(\Omega)^{d \times d})$ and its limit in $C_b(\Omega)^{d \times d}$.

We begin with a sufficient condition for the regularity of functional summation methods \eqref{eq:functional-transformation}. Unlike Theorem~\ref{thm:LinearCharacterization}, the following result is not extended from any analogous result for scalar series but \cite[Theorem~25]{hardy} comes closest.
\begin{theorem}\label{thm:functionalCharacterization}
    Let $\Omega \subseteq \mathbb{R}$, $x_* \in \overline{\Omega}$ or $x_* =\infty$, $B_k \in C_b(\Omega)^{d \times d}$, and $k\in \mathbb{N}$. Suppose 
    \begin{enumerate}[\normalfont(i)]
    \item there exists $\eta_0>0$ such that $\lVert B_0(x) \rVert \le \eta_0$ for all $x\in \Omega $; \label{cocond:1} 
    \item $\lim_{x\to x_*}B_k(x)=I$ for each $k \in \mathbb{N}$; \label{cocond:2} 
    \item there exists $\eta_1>0$ such that $\sum_{k=0}^\infty \lVert B_k(x)-B_{k+1}(x) \rVert < \eta_1$ for all $x\in \Omega $. \label{cocond:3}
    \end{enumerate}
    Then for any $A_\bullet  \in s(\mathbb{C}^{d \times d})$ such that $\sum_{k=0}^\infty A_k=S$, the series $\sum_{k=0}^\infty B_k(x)A_k$ is summable in the conventional sense for each $x\in \Omega$ with
    \begin{equation}\label{eq:functional-convergence}
\sum_{k=0}^\infty B_k(x)A_k \in C_b(\Omega)^{d \times d}  \quad \text{and} \quad     \lim_{x\to x_*}\sum_{k=0}^\infty B_k(x)A_k = S.
    \end{equation}
\end{theorem}
\begin{proof}
Let $S_n = \sum_{k=0}^n A_k$, $n \in \mathbb{N}$. Then
\[
    \sum_{k=0}^n B_k(x)A_k = \sum_{k=0}^{n-1} (B_k(x)-B_{k+1}(x))S_k + B_n(x)S_n.
\]
By Conditions~\ref{cocond:1} and \ref{cocond:3}, for each $x\in \Omega $ and $n \in \mathbb{N}$,
\[
    \lim_{k \to \infty} \lVert B_k(x)-B_{k+1}(x)\rVert = 0 , \quad \lVert B_n(x) \rVert \le \lVert B_0(x) \rVert + \sum_{k=0}^{n-1} \lVert B_k(x)-B_{k+1}(x)\rVert \le \eta_0 + \eta_1.
\]
So the sequence $(B_k(x))_{k=0}^\infty$ is convergent for each $x\in \Omega $ and 
\[
\Bigl \lVert \lim_{k\to \infty} B_k(x) \Bigr \rVert \le \eta_0 + \eta_1.
\]
Thus there exists a bounded function $B(x)$ such that $\lim_{k\to \infty} B_k(x)=B(x)$ for each $x \in \Omega$. 

We start with the left equality in \eqref{eq:functional-convergence}. As $S_\bullet$ is convergent, it is bounded by some $\beta>0$. By Condition~\ref{cocond:3},
\begin{align*}
    \sup_{x\in \Omega}\sum_{k=0}^\infty \lVert B_k(x)A_k \rVert &\leq \sup_{x\in \Omega} \biggl[\sum_{k=0}^{\infty} \lVert B_k(x)-B_{k+1}(x)\rVert \cdot \lVert S_k \rVert  + \lVert B(x) \rVert \cdot \lVert S \rVert \biggr] \\
    &\leq \sup_{x \in \Omega} \biggl[\sum_{k=0}^\infty \lVert B_k(x)-B_{k+1}(x)\rVert \cdot \lVert S_k \rVert\biggr] + \sup_{x \in \Omega}\lVert B(x) \rVert \cdot \lVert S \rVert < \infty.
\end{align*}
As absolute summability implies summability in a Banach space, we have  $\sum_{k=0}^\infty B_k(x)A_k \in C_b(\Omega)^{d \times d}$.

For the limit in \eqref{eq:functional-convergence}, assume first that $S=0$. For $\varepsilon>0$, choose $m\in \mathbb{N}$ sufficiently large so that $\norm{S_k} < \varepsilon/2\eta_1$ for all $k > m$. By Condition~\ref{cocond:2},
\[
    \lim_{x\to x_*}\sum_{k=0}^m (B_k(x)-B_{k+1}(x))S_k = 0.
\]
By Condition~\ref{cocond:3},
\[
    \Bigl\|\sum_{k\geq m}(B_k(x)-B_{k+1}(x))S_k \Bigr\| \le \frac{\varepsilon}{2\eta_1}\sum_{k\geq m} \norm{B_k(x)-B_{k+1}(x)}\le \frac{\varepsilon}{2}.
\]
Therefore,
\[
    \lim_{x\to x_*}\sum_{k=0}^\infty B_k(x)A_k = \lim_{x\to x_*} \sum_{k=0}^{\infty} (B_k(x)-B_{k+1}(x))S_k + B(x)S = 0.
\]
For $S \neq 0$, we just apply the same argument to $A'_\bullet = (A'_k)_{k=0}^\infty$ with
\[
A_k' \coloneqq \begin{cases} A_0 -S & k = 0,\\ 
A_k & k=1,2, \dots.
\end{cases} \qedhere
\]
\end{proof}
Theorem~\ref{thm:functionalCharacterization} may be interpreted as follows: A functional summation method \eqref{eq:functional-transformation} that satisfies Condition~\ref{cocond:2} is a perturbation of the original series. If  the perturbation functions $B_k$'s are uniformly bounded, i.e., Condition~\ref{cocond:1} holds, and if the changes are small enough at each step in the sense of Condition~\ref{cocond:3}, then we have regularity. We will use Theorem~\ref{thm:functionalCharacterization} to establish regularity of two powerful matrix summation methods.

\subsection{Abelian means}\label{sec:abelian}

The scalar version of this class of summation methods gained its name from the Abel summation method, which contains the well-known Abel's Theorem for power series as a special case \cite{history}. We will generalize it to matrix series.
\begin{definition}\label{def:Abemeans}
    Let $P_\bullet \coloneqq (P_k)_{k=0}^\infty$ be an unbounded sequence of positive definite matrices strictly increasing in the Loewner order, i.e., $0 \prec P_0 \prec P_1 \prec \cdots$, and $\lim_{k \to \infty} \lVert P_k \rVert = \infty$. For $A_\bullet = (A_k)_{k=0}^\infty \in s(\mathbb{C}^{d \times d})$, the series $\sum_{k=0}^\infty A_k $ is \emph{summable in Abelian means} to $S \in \mathbb{C}^{d \times d}$ with respect to $P_\bullet$ if 
    $\sum_{k=0}^\infty A_k e^{-P_k x}$ is conventionally summable for all $x\in (0,\infty)$ and 
    \[
        \lim_{x\to 0}\sum_{k=0}^\infty A_k e^{-P_k x} = S.
    \]
    We denote this by
    \[
        \sum_{k=0}^\infty A_k \doubleovereq{A}{P_\bullet} S.
    \] 
\end{definition}
\begin{corollary}[Regularity of Abelian means]
    Let $P_\bullet \coloneqq (P_k)_{k=0}^\infty$ be such that $0 \prec P_0 \prec P_1 \prec \cdots$ and $\lim_{k \to \infty} \lVert P_k \rVert = \infty$.  For any $A_\bullet \in s(\mathbb{C}^{d \times d})$, if $\sum_{k=0}^\infty A_k =S$, then $\sum_{k=0}^\infty A_k \doubleovereq{A}{P_\bullet} S$.
\end{corollary}
\begin{proof}
Summation by Abelian means with respect to $P_\bullet$ is of the form \eqref{eq:functional-transformation}. So we just need to check the conditions of Theorem~\ref{thm:functionalCharacterization}. For any $x\in (0,\infty)$,
\[
    \sum_{k=0}^\infty \bigl( e^{-P_kx} - e^{-P_{k+1}x}  \bigr) = e^{-P_0 x} -\lim_{k \to \infty} e^{-P_k x} = e^{-P_0 x} \preceq I.
\]
So Conditions~\ref{cocond:1} and \ref{cocond:3} are satisfied. Condition~\ref{cocond:2} is also satisfied  as $\lim_{x\to 0^+} e^{-P_kx} =I$ for any $k \in \mathbb{N}$.
\end{proof}

The special case  $P_\bullet = (0,I,2I,\dots)$, which reduces to a power series under the change-of-variable $t = e^{-x}$, gives us the matrix analogue of the well-known Abel summability. 
\begin{definition}
    For $A_\bullet = (A_k)_{k=0}^\infty \in s(\mathbb{C}^{d \times d})$, the series $\sum_{k=0}^\infty A_k$ is \emph{Abel summable} to $S \in \mathbb{C}^{d \times d}$ if $\sum_{k=0}^\infty A_k x^k$ is conventionally summable for all $x\in (0,1)$ and
    \[
        \lim_{x \to 1^-}\sum_{k=0}^\infty A_k x^k = S.
    \]
    We denote this by
    \[
        \sum_{k=0}^\infty A_k \overeq{A} S.
    \]
\end{definition}
We will see next that Abel summability is implied by Ces\`aro summability.
\begin{theorem} 
    For $A_\bullet \in s(\mathbb{C}^{d \times d})$ and $S\in \mathbb{C}^{d \times d}$, if $ \sum_{k=0}^\infty A_k \overeq{C} S$, then $\sum_{k=0}^\infty A_k \overeq{A} S$. 
\end{theorem}
\begin{proof}
Suppose $ \sum_{k=0}^\infty A_k \overeq{C} S$. For any $n\in \mathbb{N}$, let $S_n = \sum_{k=0}^n A_k$  and
\[
    \Sigma_n = \frac{1}{n}\sum_{k=0}^{n-1}S_k.
\]
For any $x\in (0,1)$, we have $S_n = (n+1)\Sigma_{n+1}-n\Sigma_n$ and
\begin{equation}\label{eq:abel1}
    \sum_{k=0}^n A_k x^k = S_n x^n + \sum_{k=0}^{n-1}S_k x^k(1-x). 
\end{equation} 
Since $\lim_{n\to \infty} nx^n = 0 $ and $\lim_{n\to \infty} \Sigma_n = S$, we get
$\lim_{n\to \infty} S_n x^n = 0$. By \eqref{eq:abel1},
\begin{align*}
    \biggl\|{\sum_{k=0}^\infty A_k x^k-S} \biggr\|
    &=\biggl\|{\sum_{k=0}^\infty S_k x^k (1-x)-\frac{S}{1-x}(1-x)} \biggr\|\displaybreak[0] = \biggl\|{\sum_{k=0}^\infty S_k x^k(1-x)-\sum_{k=0}^\infty S x^k(1-x)}\biggr\| \displaybreak[0]\\
    &= \biggl\|{\sum_{k=0}^\infty (S_k-S) x^k(1-x)} \biggr\| \displaybreak[0] \le \sum_{k=0}^\infty \norm{S_k-S}x^k(1-x).
\end{align*}
For $x\in (0,1)$, $\varepsilon>0$, choose $m$ sufficiently large such that if $k,p,q>m$,  
\[
    \norm{\Sigma_k - S} < \frac{(1-x)}{2} \varepsilon\quad \text{and} \quad \norm{\Sigma_p-\Sigma_q} < \frac{(1-x)}{2}\varepsilon.
\]
Then 
\begin{align*}
    \sum_{k=0}^\infty \norm{S_k-S}x^k(1-x)\displaybreak[0]   &= (1-x)\biggl[\sum_{k=0}^m x^k\norm{S_k-S} + \sum_{\mathclap{k=m+1}}^\infty x^k\norm{S_k-S}\biggr] \displaybreak[0]\\
    &= (1-x)\biggl[\sum_{k=0}^m x^k\norm{S_k-S} + \sum_{\mathclap{k=m+1}}^\infty x^k\norm{k(\Sigma_{k+1}-\Sigma_k)+(\Sigma_{k+1}-S)}\biggr] \displaybreak[0]\\
    &\le (1-x)\biggl[\sum_{k=0}^m x^k\norm{S_k-S} + \sum_{\mathclap{k=m+1} }^\infty x^k \Bigl(k\norm{\Sigma_{k+1}-\Sigma_k}+\norm{\Sigma_{k+1}-S} \Bigr)\biggr] \displaybreak[0]\\
    &< (1-x)\biggl[\sum_{k=0}^m x^k\norm{S_k-S} + \sum_{\mathclap{k=m+1}}^\infty x^k(k+1)(1-x)\varepsilon\biggr] \displaybreak[0]\\
    &= (1-x)\sum_{k=0}^m x^k\norm{S_k-S} + \varepsilon(1-x)^2\sum_{\mathclap{k=m+1}}^\infty(k+1)x^k \displaybreak[0]\\
    &< (1-x)\sum_{k=0}^m x^k\norm{S_k-S} + \varepsilon.
\end{align*}
Taking limit $x \to 1^{-}$, we deduce the required Abel summability.
\end{proof} 
Again, we will use the Neumann series as a test case. From the perspective of summing the Neumann series, Abel summation is the ``right" generalization of Ces\`aro summation in that it overcomes the difficulty associated with Jordan blocks of size greater than $1$, which we discussed after Proposition~\ref{prop:cesaro}.
\begin{lemma}
    Let $J_{\lambda} \in \mathbb{C}^{d \times d}$ be a Jordan block with eigenvalue $\lambda$. Then
    \[
    \sum_{k=1}^\infty J_{\lambda}^k \overeq{A} (I-J_{\lambda})^{-1}
\]
    if and only if $\lambda \in \overline{\mathbb{D}} \setminus \{1\}$. 
\end{lemma}
\begin{proof}
If $\abs{\lambda}>1$, then $\sum_{k=0}^\infty J_{\lambda}^k x^k$ is not summable for $x > 1/\lambda$, so the series is not Abel summable. If $\lambda=1$, then $\lim_{x\to 1^-} \sum_{k=0}^\infty J_{\lambda}^k x^k$ does not exist, so the series is not Abel summable either.

For the converse, suppose $\lambda \in \overline{\mathbb{D}} \setminus \{1\}$. Let  $0<x<1$. 
By \eqref{eq:Jordan}, for $i,j=1,\dots,d$ with $j \geq i$, the $(i,j)$th entry of the matrix
\[
    \biggl(\sum_{k=0}^\infty J_{\lambda}^k x^k \biggr)_{ij} = x^{j-i}\sum_{k=0}^\infty \binom{k+j-i}{j-i}(\lambda x)^k = \frac{x^{j-i}}{(1-\lambda x)^{j-i+1}}.
\]
Therefore, 
\[
    \sum_{k=1}^\infty J_{\lambda}^k \overeq{A}
    \begin{bmatrix}
        \frac{1}{1-\lambda} &\frac{1}{(1-\lambda)^2}  &\frac{1}{(1-\lambda)^3} &\ldots &\frac{1}{(1-\lambda)^{d-1}}\\
         &\frac{1}{1-\lambda} &\frac{1}{(1-\lambda)^2}   &\ldots &\frac{1}{(1-\lambda)^{d-2}}\\
         & &\ddots &\ddots &\vdots\\
         & & &\frac{1}{1-\lambda} &\frac{1}{(1-\lambda)^2}\\
         & & & &\frac{1}{1-\lambda}
    \end{bmatrix} = (I-J_{\lambda})^{-1}. \qedhere
\]
\end{proof}
We then apply the lemma to a Jordan decomposition to deduce what we are after.
\begin{corollary}[Abel summability of Neumann series]
    For $X\in \mathbb{C}^{d\times d}$,
\[
    \sum_{k=0}^\infty X^k \overeq{A} (I-X)^{-1}
\] 
if and only if  $\lambda(X)\subseteq \overline{\mathbb{D}} \setminus \{1\}$.
\end{corollary}
\begin{proof}
Let $\lambda(X) = \{ \lambda_1,\dots, \lambda_r\} \subseteq \overline{\mathbb{D}} \setminus \{1\}$ counting multiplicities and $X = WJW^{-1}$ be a Jordan decomposition with  $J = \diag(J_{\lambda_1},\dots,J_{\lambda_r})$ and $J_{\lambda_i}$ the Jordan block corresponding to $\lambda_i$, $i=1,\dots, r$. Then
\[
\sum_{k=0}^\infty X^k =\sum_{k=0}^\infty W 
\begin{bmatrix}
    J_{\lambda_1}^k & &\\
     &\ddots &\\
     & &J_{\lambda_r}^k 
\end{bmatrix}W^{-1} \overeq{A}  W
\begin{bmatrix}
    (I-J_{\lambda_1})^{-1} & &\\
     &\ddots &\\
     & &(I-J_{\lambda_r})^{-1}
\end{bmatrix}W^{-1} = (I-X)^{-1}.
\]
For the converse, suppose without loss of generality that $\lambda_1 \notin \overline{\mathbb{D}} \setminus \{1\}$. Then
\[
    \sum_{k=0}^\infty W 
\begin{bmatrix}
    J_{\lambda_1}^k & &\\
     &\ddots &\\
     & &J_{\lambda_r}^k 
\end{bmatrix}W^{-1}
\]
is not Abel summable as the first Jordan subblock is not Abel summable.
\end{proof}

\subsection{Lambert method}\label{sec:lambert}

The original Lambert summation method, named after Johann Heinrich Lambert, played an important role in number theory \cite{hardylittlewood21, tauberian_theory, wiener32}, and is a particularly potent tool for summing Dirichlet series, as we will see in Section~\ref{sec:dirichlet}.  Here we will generalize Lambert summation to matrix series. 
\begin{definition}\label{def:lambert}
   For $A_\bullet = (A_k)_{k=0}^\infty \in s(\mathbb{C}^{d \times d})$, the series $\sum_{k=1}^\infty A_k$ is \emph{Lambert summable} to $S\in \mathbb{C}^{d \times d}$ if $\sum_{k=1}^\infty {kA_kx^k}/(1+x+\dots +x^{k-1})$ is conventionally summable for every $x\in(0,1)$ and
    \[
        \lim_{x\to 1^-}(1-x)\sum_{k=1}^\infty \frac{kx^k}{1-x^k}  A_k = S.
    \]
    We denote this by
    \[
        \sum_{k=1}^\infty A_k \overeq{L} S.
    \] 
\end{definition}
\begin{corollary}[Regularity of Lambert summation]
     For $A_\bullet \in s(\mathbb{C}^{d \times d})$ and $S\in \mathbb{C}^{d \times d}$, if $\sum_{k=0}^\infty A_k =S$, then $\sum_{k=0}^\infty A_k \overeq{L} S$.
\end{corollary}
\begin{proof}
Lambert summation is of the form \eqref{eq:functional-transformation}, so we check the conditions of Theorem~\ref{thm:functionalCharacterization}. Since $\lvert x \rvert <1 $ for $x\in (0,1)$, Condition~\ref{cocond:1} is satisfied. For each $k\in \mathbb{N}$,
\[
    \lim_{x\to 1^{-}} \frac{kx^k}{1+x+\dotsb+x^{k-1}} = 1,
\]
so Condition~\ref{cocond:2} is satisfied. Condition~\ref{cocond:3} is also satisfied as for any $x\in (0,1)$, 
\[
    \sum_{k=0}^\infty (1-x)\biggl|\frac{kx^k}{1-x^k}-\frac{(k+1)x^{k+1}}{1-x^{k+1}}\biggr| = x \le 1. \qedhere
\]
\end{proof}

\subsection{Borel and Mittag-Leffler methods}\label{sec:borel}

The scalar versions of Borel summation methods, named after \'{E}mile Borel \cite{Borel_paper}, have important applications in physics \cite{glimmjaffe, borel}. They come in two variants (weak and strong) and we will extend them to matrix series.
\begin{definition}\label{def:wborel}
For $A_\bullet = (A_k)_{k=0}^\infty \in s(\mathbb{C}^{d \times d})$, its \emph{weak Borel transform} is 
\[
    \WB (A_\bullet)(x) = \sum_{k=0}^{\infty} S_k \frac{x^k}{k!}
\]
for all $x > 0$, where $S_n = \sum_{k=0}^n A_k$.
The series $\sum_{k=0}^\infty A_k$ is \emph{weakly Borel summable} to $S \in \mathbb{C}^{d \times d}$ if $e^{-x} \WB (A_\bullet)(x)$ is conventionally summable for all $x > 0$ and
\[
    \lim_{x\to \infty} e^{-x} \WB (A_\bullet)(x)=S.
\]
We denote this by
\[
    \sum_{k=0}^\infty A_k \overeq{WB} S.
\] 
\end{definition}
\begin{theorem}[Regularity of weak Borel summation]
     For $A_\bullet \in s(\mathbb{C}^{d \times d})$, if $\sum_{k=0}^\infty A_k =S$, then $\sum_{k=0}^\infty A_k \overeq{WB} S$.
\end{theorem}
\begin{proof}
Let $S_n = \sum_{k=0}^n A_k$, $n \in \mathbb{N}$. As $S_\bullet$ is a convergent sequence, it is bounded by some $\beta > 0$. For each $x > 0$,
\[
    \lVert e^{-x} \WB (A_\bullet)(x) \rVert = \biggl\lVert e^{-x} \sum_{k=0}^{\infty} \frac{x^k}{k!}S_k \biggr\rVert \le \beta  e^{-x} \sum_{k=0}^{\infty} \frac{x^k}{k!} = \beta.
\]
Hence $e^{-x} \sum_{k=0}^{\infty} S_k x^k/k!$ is conventionally summable for all $x > 0$. Moreover, 
\begin{align*}
     \bigl\| \lim_{x \to \infty} e^{-x} \WB (A_\bullet)(x) - S \bigr\| & = \biggl\| \lim_{x\to \infty} e^{-x} \sum_{k=0}^{\infty} \frac{x^k}{k!}S_k - \lim_{x\to \infty} e^{-x} \sum_{k=0}^{\infty} \frac{x^k}{k!}S \biggr\|\\
    & \le \lim_{x\to \infty} e^{-x} \sum_{k=0}^{\infty} \frac{x^k}{k!} \norm{S_k- S}= 0. \qedhere
\end{align*}
\end{proof}
\begin{definition}\label{def:sborel}
For $A_\bullet = (A_k)_{k=0}^\infty \in s(\mathbb{C}^{d \times d})$, its \emph{strong Borel transform} is
\[
\SB (A_\bullet)(x)\coloneqq \sum_{k=0}^{\infty} A_k \frac{x^k}{k!}
\]
for $x > 0$.
The series $\sum_{k=0}^\infty A_k$ is \emph{strongly Borel summable} to $S\in \mathbb{C}^{d \times d}$ if 
\[
\int_{0}^{x}  e^{-t}\SB (A_\bullet)(t) \,dt
\]
is conventionally summable for all $x>0$ and
\begin{equation}\label{eq:strongborel}
    \int_{0}^{\infty} e^{-x} \SB (A_\bullet)(x) \,dx = S.
\end{equation}
We denote this by
\[
    \sum_{k=0}^\infty A_k \overeq{SB} S.
\] 
\end{definition}
It turns out that, for a series of scalars, the strong Borel method is a special case of the Mittag-Leffler summation. We will show that the same is true for a series of matrices with the following matrix generalization of the latter.
\begin{definition}\label{def:ml}
For $A_\bullet = (A_k)_{k=0}^\infty \in s(\mathbb{C}^{d \times d})$, the series $\sum_{k=0}^\infty A_k$ is \emph{Mittag-Leffler summable} to $S\in \mathbb{C}^{d \times d}$ with respect to $\alpha > 0$ if 
\[
\int_{0}^{x}  e^{-t}\sum_{k=0}^\infty \frac{A_k t^{\alpha k}}{\Gamma(1+\alpha k)} \,dt
\]
is conventionally summable for all $x>0$ and
\[
    \int_{0}^{\infty} e^{-x} \sum_{k=0}^\infty \frac{A_k x^{\alpha k}}{\Gamma(1+\alpha k)} \,dx =S.
\]
We denote this by
\[
    \sum_{k=0}^\infty A_k \overeq{M} S.
\] 
The implicit choice of $\alpha$ is not reflected in the notation. Here $\Gamma(x) \coloneqq \int_0^\infty t^{x-1}e^{-t}\, dt$  is the  Gamma function.
\end{definition}
If we set $\alpha = 1$ above, we obtain the strong Borel method.
\begin{theorem}[Regularity of Mittag-Leffler summation]
     For $\alpha>0$ and $A_\bullet \in s(\mathbb{C}^{d \times d})$, if $\sum_{k=0}^\infty A_k =S$, then $\sum_{k=0}^\infty A_k \overeq{M} S$. In particular, if $\sum_{k=0}^\infty A_k =S$, then $\sum_{k=0}^\infty A_k \overeq{SB} S$.
\end{theorem}
\begin{proof}
This follows from
\[
    S= \sum_{k=0}^{\infty} A_k = \sum_{k=0}^{\infty} \biggl(\int_{0}^{\infty} e^{-x}x^{\alpha k} \,dx \biggr)\frac{A_k}{\Gamma(1+\alpha k)} = \int_{0}^{\infty} e^{-x} \sum_{k=0}^\infty \frac{A_k x^{\alpha k}}{\Gamma(1+\alpha k)} \,dx.  \qedhere
\]
\end{proof}
We will next justify the `weak' and `strong' designations and see when they are equivalent \cite{borweinshawyer66,borel}, generalizing \cite[Theorem~123]{hardy} to matrices.
\begin{theorem}\label{thm:borels}
Let $A_\bullet = (A_k)_{k=0}^\infty \in s(\mathbb{C}^{d \times d})$. If $\sum_{k=0}^\infty A_k \overeq{WB} S$, then $\sum_{k=0}^\infty A_k \overeq{SB} S$. The converse holds if and only if $\lim_{x\to \infty}e^{-x}\SB(A_\bullet)(x) = 0$.
\end{theorem}
For easy reference, we reproduce two lemmas used in the proof of \cite[Theorem~122]{hardy}.
\begin{lemma}[Hardy]\label{lemma:hardy}
    Let $f:\mathbb{R}\to \mathbb{C}$ be differentiable. If $\lim_{x\to \infty}f(x)+f'(x)= 0$, then $\lim_{x\to \infty} f(x)= 0$.
\end{lemma}
\begin{lemma}[Hardy]\label{lemma:transform}
Let $a_\bullet \in s(\mathbb{C})$. The series $\WB (a_\bullet)(x)$ is conventionally summable for all $x > 0$ if and only if $\SB (a_\bullet)(x)$ is conventionally summable for all $x > 0$.
\end{lemma}
\begin{proof}[Proof of Thoerem~\ref{thm:borels}]
Suppose $\sum_{k=0}^\infty A_k \overeq{WB} S$. Then $\WB(A_\bullet)(x)$ is conventionally summable for all $x > 0$. Applying Lemma~\ref{lemma:transform} entrywise shows that $\SB(A_\bullet)(x)$ is conventionally summable for all $x > 0$. Taking derivative,
\[
    \SB(A_\bullet)'(x) = \sum_{k=0}^\infty A_{k+1} \frac{x^k}{k!} \quad \text{and} \quad \WB(A_\bullet)'(x) = \sum_{k=0}^\infty S_{k+1} \frac{x^k}{k!}.
\]
Therefore,
\begin{align*}
e^{-x} \WB(A_\bullet)(x) - A_0 &= \int_{0}^x \frac{d}{dt} \bigl( e^{-t} \WB(A_\bullet)(t) \bigr)\, dt = \int_{0}^x e^{-t} \bigl( \WB(A_\bullet)'(t) - \WB(A_\bullet)(t) \bigr)\, dt\\
&= \int_{0}^x \sum_{k=0}^\infty (S_{k+1}-S_k) \frac{e^{-t}t^k}{k!}\, dt = \int_{0}^x \sum_{k=0}^\infty A_{k+1} \frac{e^{-t} t^k}{k!}\, dt\\
&=\int_{0}^x e^{-t} \SB(A_\bullet)'(t) \,dt = e^{-x} \SB(A_\bullet)(x) - A_0 + \int_{0}^x e^{-t} \SB(A_\bullet)(t)\, dt.
\end{align*}
Rearranging terms,
\begin{equation}\label{eq:weak_vs_strong}
    e^{-x} \WB(A_\bullet)(x) = \int_{0}^x e^{-t} \SB(A_\bullet)(t)\, dt + e^{-x} \SB(A_\bullet)(x).
\end{equation}
Taking limit $x\to \infty$, the left-hand side gives the weak Borel sum while the first term on the right-hand side gives the strong Borel sum. Since
\begin{align*}
\lim_{x\to\infty} e^{-x} \WB(A_\bullet)(x) &= \lim_{x\to\infty} \biggl( \int_{0}^x e^{-t} \SB(A_\bullet)(t)\, dt + e^{-x} \SB(A_\bullet)(x) \biggr)\\
&= \lim_{x\to\infty} \biggl( \int_{0}^x e^{-t} \SB(A_\bullet)(t)\, dt + \frac{d}{dx} \int_{0}^x e^{-t} \SB(A_\bullet)(t)\, dt \biggr)= S, 
\end{align*}
we may apply Lemma~\ref{lemma:hardy} entrywise so that
\[
    \lim_{x\to \infty} e^{-x} \SB(A_\bullet)(x) = 0 \quad \text{and} \quad  \lim_{x\to\infty} \int_{0}^x e^{-t} \SB(A_\bullet)(t)\, dt =S.
\]

Suppose $\sum_{k=0}^\infty A_k \overeq{SB} S$. Then $\SB(A_\bullet)(x)$ is conventionally summable for all $x > 0$ and so is $\WB(A_\bullet)(x)$ by applying Lemma~\ref{lemma:transform} entrywise. By \eqref{eq:weak_vs_strong},
\begin{align*}
\lim_{x\to \infty} e^{-x} \WB(A_\bullet)(x) &= \lim_{x\to \infty} \int_{0}^x e^{-t} \SB(A_\bullet)(t)\, dt + \lim_{x\to \infty} e^{-x} \SB(A_\bullet)(x) \\
&= S + \lim_{x\to \infty} e^{-x} \SB(A_\bullet)(x).
\end{align*}
Hence $\sum_{k=0}^\infty A_k \overeq{WB} S$ if and only if $\lim_{x\to \infty} e^{-x} \SB(A_\bullet)(x) = 0. $
\end{proof}
As we mentioned at the end of Section~\ref{sec:euler}, both Borel methods generalize the Euler methods.
\begin{theorem}
Let $A_\bullet = (A_k)_{k=0}^\infty \in s(\mathbb{C}^{d \times d})$ and $P \in \mathbb{C}^{d \times d}$ be such that $P \succ 0$. If $\sum_{k=0}^\infty A_k \doubleovereq{E}{P} S$, then $\sum_{k=0}^\infty A_k \overeq{WB} S$ and $\sum_{k=0}^\infty A_k \overeq{SB} S$.
\end{theorem}
\begin{proof}
Let $S_n = \sum_{k=0}^n A_k$, $n\in \mathbb{N}$. By definition of Euler summability, $\sum_{k=0}^\infty A_k \doubleovereq{E}{P} S$ if and only if $\lim_{n\to \infty} Z_n = S$ where
\[
    Z_n = \sum_{k=0}^n \E_k^P(A_\bullet) = (I+P)^{-n-1}\sum_{k=0}^n \binom{n+1}{k+1} P^{n-k} S_k.
\]
Then
\begin{align*}
e^{Px} \sum_{k=0}^\infty S_k \frac{x^k}{k!} &= \Biggl[ \sum_{k=0}^\infty \frac{(Px)^k}{k!} \Biggr] \Biggl[ \sum_{k=0}^\infty S_k \frac{x^k}{k!} \Biggr] = \sum_{k=0}^\infty \biggl[ \frac{S_k}{k!} + \frac{P S_{k-1}}{(k-1)!} + \frac{P^2 S_{k-2}}{(k-2)! 2!} + \dots + \frac{P^k S_0}{k!} \biggr] x^k\\
&= \sum_{k=0}^\infty \frac{(I+P)^k x^k}{k!} Z_k.
\end{align*}
Thus,
\begin{equation}\label{eq:Borel-eq}
    e^{-x} \sum_{k=0}^{\infty} \frac{x^k}{k!} S_k = e^{-(I+P)x} \sum_{k=0}^\infty \frac{(I+P)^k x^k}{k!} Z_k.
\end{equation}
Weak Borel summability follows as
\begin{align*}
     \Bigl\|\lim_{x \to \infty} e^{-x} \sum_{k=0}^{\infty} \frac{x^k}{k!}S_k - S \Bigr\| &=\Bigl\| \lim_{x \to \infty} e^{-(I+P)x}\sum_{k=0}^\infty \frac{(I+P)^k x^k}{k!} Z_k - \lim_{x \to \infty} e^{-(I+P)x}\sum_{k=0}^\infty \frac{(I+P)^k x^k}{k!} S \Bigr\|\\
    & \le \lim_{x \to \infty} e^{-(I+P)x}\sum_{k=0}^\infty \frac{(I+P)^k x^k}{k!} \lVert Z_k - S \rVert = 0. 
\end{align*}
By Theorem~\ref{thm:borels}, we obtain $\sum_{k=0}^\infty A_k \overeq{SB} S$.
\end{proof}
As before we will use the Neumann series as our basic test case. The proof below sheds further light on how Borel summation generalize Euler summation.
\begin{proposition}[Borel summability of Neumann series]\label{prop:borneu}
For $X\in \mathbb{C}^{d\times d}$, the following are equivalent:
\begin{enumerate}[\normalfont(i)]
    \item $\sum_{k=0}^\infty X^k \overeq{WB} (I-X)^{-1}$; \label{BorelNeumann1}
    \item $\sum_{k=0}^\infty X^k \overeq{SB} (I-X)^{-1}$; \label{BorelNeumann2}
    \item $\lambda(X) \subseteq \{z \in \mathbb{C} :  \re(\lambda) < 1 \}$. \label{BorelNeumann3}
\end{enumerate}
\end{proposition}
\begin{proof}
Let $\lambda(X) =\{\lambda_1,\dots, \lambda_r\}$. We show that \ref{BorelNeumann2} and \ref{BorelNeumann3} are equivalent. If $1\in \lambda(X)$, then the integral
\[
    \int_{0}^{\infty} e^{-x} \sum_{k=0}^{\infty} \frac{(xX)^k}{k!} \, dx
\]
is divergent, so the Neumann series $\sum_{k=0}^\infty X^k$ is not strongly Borel summable. If $1\notin \lambda(X)$, then
\begin{equation}\label{eq:SB_Neumann}
    \int_{0}^{\infty} e^{-x} \sum_{k=0}^{\infty} \frac{(xX)^k}{k!} \, dx = \int_{0}^{\infty} e^{x(X-I)} \, dx= (X-I)^{-1} \Bigl( \lim_{x \to \infty} e^{x(X-I)} -I \Bigr),
\end{equation}
so the Neumann series is strongly Borel summable to $(I-X)^{-1}$ if and only if
\begin{equation}\label{eq:eto0}
     \lim_{n \to \infty} e^{n(X-I)} = 0.
\end{equation}
As $\lim_{n \to \infty} A^n = 0$ if and only if $\lambda(A) \subseteq \mathbb{D}$; $\lvert e^{\lambda-1} \rvert  <1$ if and only if $\re(\lambda)<1$; and $\lambda(e^{X-I}) = \{e^{\lambda_1-1},\dots, e^{\lambda_r-1}\}$; we deduce that \eqref{eq:eto0} holds if and only if \ref{BorelNeumann3} holds.

We next show that \ref{BorelNeumann1} and \ref{BorelNeumann3} are equivalent. The geometric series $\sum_{k=0}^\infty \lambda^k$ is not weakly Borel summable at $\lambda=1$ as
\[
    \lim_{x \to \infty} e^{-x}\sum_{k=0}^\infty \frac{x^k}{(k-1)!} = \lim_{x \to \infty} e^{-x}\frac{d}{dx} (xe^x) = \lim_{x \to \infty} 1+x = + \infty.
\]
So the Neumann series is not weakly Borel summable if $1\in \lambda(X)$. If $1\notin \lambda(X)$, then
\begin{align*}
    \lim_{x\to \infty} e^{-x}\sum_{k=0}^\infty (I-X)^{-1} (I-X^{k+1} ) \frac{x^k}{k!} 
    &= \lim_{x\to \infty} e^{-x}(I-X)^{-1} \sum_{k=0}^\infty (I-X^{k+1} ) \frac{x^k}{k!} \\
    &= \lim_{x\to \infty} e^{-x}(I-X)^{-1} (e^x I - Xe^{xX} )\\
    &= (I-X)^{-1} - \lim_{x\to \infty}Xe^{x(X-I)}.
\end{align*}
As in the case of \eqref{eq:eto0}, the last limit is zero if and only if \ref{BorelNeumann3} holds.
\end{proof}
In this context, Borel summation may be viewed as a limiting case of Euler summation as $\rho \to \infty$: By Corollary~\ref{cor:euneu}, $\sum_{k=0}^\infty X^k \doubleovereq{E}{\rho} (I-X)^{-1}$ if and only if $\lambda(X) \subseteq \{ z \in \mathbb{C} : \lvert z + \rho \rvert < 1 + \rho \}$; and
\[
    \bigcup_{\rho>0} \{ z \in \mathbb{C} : \lvert z + \rho \rvert < 1+\rho\} = \{z \in \mathbb{C} :  \re(\lambda)< 1\}.
\] 

\section{From theory to computations}\label{sec:compute}

In principle, every summation method discussed in Sections~\ref{sec:seq} and \ref{sec:func} yields a numerical method for summing a matrix series. But in reality issues related to rounding errors  will play an important role and have to be carefully treated.  We will discuss these in the next two sections after making some observations.

The theoretical results in Sections~\ref{sec:seq} and \ref{sec:func} are all about convergence (whether a method converges or diverges) but say nothing about the \emph{rate} of convergence. Readers familiar with the matrix functions literature \cite{funcofmat} may think that  convergence rates should be readily obtainable but this is an illusion---the matrix series appearing in the matrix functions literature are invariably \emph{power} or \emph{Taylor series}, whereas the series appearing in Sections~\ref{sec:seq} and \ref{sec:func} can be any arbitrary matrix series---Fourier, Dirichlet, Hadamard powers, etc.

For special matrix series whose $k$th term takes a simple fixed form like $X^k$, $\sin (kX)$, $\exp(X \log k)$, $X^{\circ k}$, there is some hope of deriving a `remainder' that gives the convergence rate but even that may be a difficult undertaking. For an arbitrary matrix series $\sum_{k=0}^\infty A_k$, where the $k$th term can be any matrix $A_k \in \mathbb{C}^{d \times d}$, such `remainder' do not generally exist even when it is a scalar series \cite{boos, hardy, peyerimhoff69, borel}.

To sum an arbitrary matrix series of complete generality, we may thus only assume that the truncated sum $\sum_{k=0}^n A_k \approx \sum_{k=0}^\infty A_k$ is ascertained to be a good approximation through some other means and that $n$ is given as part of the input---there is no `remainder' that allows one to estimate $n$ a priori. We will have more to say about this issue in Section~\ref{sec:gen}, where we will also discuss matrix adaptations of compensated summation \cite{KahanSum}, block and mixed block summations \cite{HighamBlanchard}, methods that were originally developed for sums of scalars.

The special case of Taylor or power series, i.e., where $A_k = c_k A^k$, deserves special attention because of their central role in matrix functions \cite{funcofmat}. In Section~\ref{sec:pow}, we discuss how one may adapt the Pad\'e approximation \cite{scaling-squaring} and Schur--Parlett \cite{Parlett} algorithms to work with any of the regular sequential summation methods $\mathsf{R}$ in Section~\ref{sec:seq}, i.e., compute the $\mathsf{R}$-sum $S \overeq{R} \sum_{k=0}^\infty c_k A^k$ for a given $A \in \mathbb{C}^{d \times d}$ using these algorithms.

Henceforth we assume the standard model for floating-point arithmetic \cite[Section~2.2]{Higham_Accuracy}:
\[
    \fl(a \ast  b) = (a \ast  b)(1+\delta),\quad   \abs{\delta}\le \ur,\quad \ast   \in \{+, -, \times, \div \},
\]
with $\fl(a)$ the computed value of $a\in \mathbb{R}$ in floating-point arithmetic and $\ur$  the unit roundoff.  For any computations in floating-point arithmetic involving more than a single operation, we denote  by $\widehat{S}$ the final computed output of a quantity $S$. This is to avoid having to write, say, $\fl(a + \fl(b + \fl (c + d)))$, for the output of $a+b +c+d$, unless it is strictly necessary (like in Algorithm~\ref{alg:comp}).

\section{Accurate and fast numerical summation}\label{sec:gen}

In this section there will be no loss of generality in restricting our discussions to $\mathbb{R}$, since complex addition is performed separately for real and imaginary parts as real additions. As we alluded to in Section~\ref{sec:compute}, for a general matrix series $S = \sum_{k=0}^\infty A_k$ where $A_k$ has no special form, computing it means to approximate $S$ up to some desired $\varepsilon$-accuracy by a partial sum $S_n = \sum_{k=0}^n A_k$, i.e., with $\lVert S_n - S \rVert < \varepsilon$. There are two considerations in choosing $n \in \mathbb{N}$.

Firstly, for a given $\varepsilon > 0$, the value of $n$ depends on the summation method we choose. This is in fact an important motivation for the summation methods in Sections~\ref{sec:seq} and \ref{sec:func}, namely, they often require a smaller $n$ to achieve the same $\varepsilon$-accuracy. For example, take any scalar alternating series $\sum_{k=0}^\infty a_k = s$ that is conventionally summable; it is known \cite{Rosser} that for 
\[
    \biggl \lvert s -\sum_{k=0}^{n_1} a_k \biggr \rvert < \varepsilon, \qquad     \biggl \lvert s -\sum_{k=0}^{n_2} \E_k^1(a_\bullet) \biggr \rvert < \varepsilon,
\]
we  need only $n_2 < n_1$ terms. Here $\E_k^1(a_\bullet)$ is the $1$-Euler transform as defined in \eqref{eq:Aq}. In other words, Euler summation gets us to the same $\varepsilon$-accuracy with fewer terms than conventional summation. This advantage extends to series of matrices, as we will see with the Neumann series in Section~\ref{sec:acc}.

Secondly, for a fixed choice of summation method and a fixed $\varepsilon > 0$, the value of $n$ is highly sensitive to the order of summation and termination criteria. This is already evident in conventional summation of scalar series $s = \sum_{k=0}^\infty a_k$. Clearly we could not rely on $\lvert s_n - s \rvert = \lvert \sum_{k=n+1}^\infty a_k \rvert <\varepsilon$ as a termination criterion since the value of $s$ is precisely what needs to be determined.

Suppose  we use $\lvert a_k \rvert< \varepsilon$ (using $\lvert a_k \rvert/\lvert s_k \rvert< \varepsilon$ would not make much of a difference) as termination criterion with  $\varepsilon = 10^{-6}$ and we use the geometric series with $a_k = 2^{-k}$ for illustration since we know $s = 2$. A straightforward summation algorithm is given by setting $s \leftarrow a_0 = 1/2$ and iteratively computing
\begin{equation}\label{eq:naive}
    s \leftarrow s + a_k \quad \text{for} \quad k=1, \dotsc, n,
\end{equation}
until $\lvert a_n \rvert < 10^{-6}$, which gives the correct answer $\widehat{s} = 2$ in single precision. However, if we apply the same algorithm to what is essentially the same series with a single zero added as the first term:
\[
b_k=\begin{cases}
0  & k = 0,\\
a_{k-1} & k \ge 1,
\end{cases}
\]
then although $s = \sum_{k=0}^\infty b_k = \sum_{k=0}^\infty a_k$, the computed sum is now $\widehat{s} = 0$ as it terminates at $n=0$. The bottom line is that there is no universal termination criterion---$n$ has to be ascertained on a case-by-case basis and for a general series we will have to assume that it is given as part of our input. Henceforth we will assume this and our goal is to compute $s \coloneqq s_n = \sum_{k=1}^n a_k$ accurately.

The na\"ive algorithm in \eqref{eq:naive} is called \emph{recursive summation} \cite[Section~4.1]{Higham_Accuracy}. It computes $s$ with an error given by
\begin{equation}\label{eq:recursive}
    \widehat{s}=\sum_{k=0}^n a_k(1+\delta_k), \quad \abs{\delta_k} \le n \ur + O\bigl(\ur^2\bigr).
\end{equation}
The algorithm extends immediately to matrix sums  $S = \sum_{k=1}^n A_k \in \mathbb{R}^{d \times d}$. Since matrix addition is computed entrywise, if we write $s_{ij}$ and $a_{ijk}$ for the $(i,j)$th entry of $S$ and $A_k$ respectively, then \eqref{eq:recursive} generalizes to
\[
    \widehat{s}_{ij}=\sum_{k=0}^n a_{ijk} \bigl(1+ \delta_{ijk} \bigr), \quad \abs{\delta_{ijk}} \le n \ur +O\bigl(\ur^2\bigr), \quad i,j = 1, \dots, d,
\]
or, in terms of the Hadamard product $\circ$ and  writing $\Delta_k \coloneqq (\delta_{ijk})  \in \mathbb{R}^{d \times d}$,
\begin{equation}\label{eq:naive1}
    \widehat{S} =  \sum_{k=0}^n A_k \circ (\mathbbm{1}+\Delta_k ), \quad \abs{\delta_{ijk}} \le n \ur+O\bigl(\ur^2\bigr), \quad i,j = 1, \dots, d.
\end{equation}
Since for $A,B \in \mathbb{R}^{d \times d}$,
\begin{equation}\label{eq:Hadamard}
    \norm{A \circ B}\le \norm{A} \max_{i,j=1,\dots,d} \lvert b_{ij} \rvert,
\end{equation}
we obtain the forward error bound
\[
    \norm{S-\widehat{S}} \le \sum_{k=0}^n \norm{A_k\circ \Delta_k} \le n \ur \sum_{k=0}^n \norm{A_k} +O\bigl(\ur^2\bigr) .
\]
This serves as a baseline bound---we will discuss three more accurate summation methods that can significantly reduce the coefficient $n \ur$ to $O(\sqrt{n}\ur)$, $O(\ur)$, and even $O(\ur^2) $.

For a scalar series, a simple strategy  \cite[Section~4.2]{Higham_Accuracy} to improve accuracy of \eqref{eq:naive} is to reorder the summands in increasing magnitudes to minimize the rounding error at each step. Note that this does not work for matrix series since there is no natural total order on $\mathbb{R}^{d \times d}$ and reordering often improves the accuracy of one entry at the expense of decreased accuracy in another.

\subsection{Block summation algorithm}\label{sec:block}

Assume without loss of generality that  $b \in \mathbb{N}$ divides $n+1$. The block summation algorithm \cite[Section~2.2]{HighamBlanchard} in Algorithm~\ref{alg:block} modifies recursive summation \eqref{eq:naive} by dividing the sum into blocks of size $b$. In particular, it allows the block sums to be computed in parallel.

\begin{algorithm}
\caption{Block summation}\label{alg:block}
\begin{algorithmic}[1]
\Require $A_0,\dotsc, A_n \in \mathbb{R}^{d \times d}$, block size $b$;
\For{$ k = 1,\dotsc,(n+1)/b$}
    \State compute $S_i=\sum_{k=(i-1) b}^{i b-1} A_k$ with recursive summation \eqref{eq:naive};
\EndFor
\State compute $S = \sum_{i=1}^{(n+1)/b} S_i$ with the recursive summation \eqref{eq:naive};
\Ensure $S$.
\end{algorithmic}
\end{algorithm}

As in our discussion of \eqref{eq:naive}, it is straightforward to extend \cite[Equation~2.4]{HighamBlanchard} to a sum of matrices: Algorithm~\ref{alg:block} satisfies
\[
    \widehat{S} =  \sum_{k=0}^n A_k \circ (\mathbbm{1}+\Delta_k ), \quad \abs{\delta_{ijk}} \le \Bigl(b+\frac{n+1}{b}-2 \Bigr) \ur +O\bigl(\ur^2\bigr),\quad  i,j = 1, \dots, d,
\]
with notations as in \eqref{eq:naive1}. By \eqref{eq:Hadamard}, we obtain the forward error bound for Algorithm~\ref{alg:block}:
\[
    \norm{S-\widehat{S}} \le \Bigl(b+\frac{n+1}{b}-2 \Bigr) \ur \sum_{k=0}^n \norm{A_k}+O\bigl(\ur^2\bigr).
\]
The optimal bound $2\sqrt{n+1}-2$ is easily seen to be attained with $b=\sqrt{n+1}$ although in practice it is common to choose $b$ to be a constant such as $128$ or $256$.

The parallelism in Algorithm~\ref{alg:block} requires summands to be independent and may be lost in situations like computing a matrix polynomial $\sum_{k=0}^n c_k A^k$ with Horner's method (Algorithm~\ref{alg:horner}).

\subsection{Compensated summation algorithm}\label{sec:comp}

This is also known as Kahan summation \cite{KahanSum} and is based on a clever exploitation of the floating point system. By observing that the rounding error in a floating-point addition of two matrices is itself a floating-point matrix, Algorithm~\ref{alg:comp} simply approximates this error with a correction term $C \in \mathbb{R}^{d \times d}$ at each step of recursive summation to adjust the computed sum.

\begin{algorithm}
\caption{Compensated summation}\label{alg:comp}
\begin{algorithmic}[1]
\Require $A_0,\dotsc, A_n \in \mathbb{R}^{d \times d}$;
\State initialize $S \leftarrow 0$, $C \leftarrow 0$;
\For{$ k = 0,\dotsc, n$}
    \State $Y \leftarrow \fl(A_k - C)$;
    \State $T \leftarrow \fl(S + Y)$;
    \State $C \leftarrow \fl(\fl(T - S) - Y)$;
    \State $S \leftarrow T$;
\EndFor
\Ensure $S$.
\end{algorithmic}
\end{algorithm}

Since the rounding error in floating point arithmetic is, by definition,  the unit round-off $\ur$, a straightforward matrix adaptation of  \cite[Equation~4.8]{Higham_Accuracy} for Algorithm~\ref{alg:comp} yields
\begin{equation}\label{eq:compensated_bound}
    \widehat{S} =  \sum_{k=0}^n A_k \circ (\mathbbm{1}+\Delta_k), \quad \abs{\delta_{ijk}} \le 2\ur+O\bigl(\ur^2\bigr), \quad i,j = 1, \dots, d,
\end{equation}
with notations as in \eqref{eq:naive1}. By \eqref{eq:Hadamard}, we obtain the forward error bound for Algorithm~\ref{alg:comp}:
\[
    \norm{S-\widehat{S}} \le 2\ur \sum_{k=0}^n \norm{A_k}+O\bigl(\ur^2\bigr).
\]
Remarkably, Algorithm~\ref{alg:comp} eliminates $n$ from the error bound. This enhanced accuracy is achieved at the cost of three extra matrix additions per loop, and is often more expensive than simply switching to higher precision \cite{HighamBlanchard}. So compensated summation  is usually deployed only when computations are already taking place at the highest available precision.

\subsection{Mixed block summation algorithm}

Assume without loss of generality that  $b \in \mathbb{N}$ divides $n+1$.  Algorithm~\ref{alg:mix} is a variant of  Algorithm~\ref{alg:block} that strikes a balance between a fast algorithm \textsc{FastSum} and an accurate algorithm \textsc{AccurateSum}.

\begin{algorithm}
\caption{Mixed block summation} \label{alg:mix}
\begin{algorithmic}[1]
\Require $A_0,\dotsc, A_n \in \mathbb{R}^{d \times d}$, block size $b$, \textsc{FastSum}, \textsc{AccurateSum};
\For{$ k = 1,\dotsc,(n+1)/b$}
    \State compute $S_i=\sum_{k=(i-1) b}^{i b-1} A_k$ with \textsc{FastSum};
\EndFor
\State compute $S = \sum_{i=1}^{(n+1)/b} S_i$ with \textsc{AccurateSum};
\Ensure $S$.
\end{algorithmic}
\end{algorithm}

When $b=1$, Algorithm~\ref{alg:mix} is exactly \textsc{AccurateSum} and when $b=n+1$, Algorithm~\ref{alg:mix} is exactly \textsc{FastSum}. The scalar version of this algorithm was proposed by Blanchard, Higham, and Mary in \cite{HighamBlanchard} and we merely adapted it for matrices. The following corollary of \cite[Theorem~3.1]{HighamBlanchard} follows from the same arguments used in Sections~\ref{sec:block} and \ref{sec:comp}. Recall that we write $\Delta_k \coloneqq (\delta_{ijk})  \in \mathbb{R}^{d \times d}$.

\begin{corollary}[Error bound of mixed block summation algorithm]
Let the sum computed with \textsc{FastSum} satisfy
\[
    \widehat{S}^{\method{F}} =  \sum_{k=0}^n A_k^{\method{F}} \circ \bigl(\mathbbm{1}+\Delta_k^{\method{F}} \bigr), \quad \abs{\delta^{\method{F}}_{ijk}} \le  \varepsilon^{\method{F}}(n), \quad i,j = 1, \dots, d,
\]
and likewise for \textsc{AccurateSum} with \textsf{A} in place of \textsf{F} in the superscript. Then the sum computed  with Algorithm~\ref{alg:mix} satisfies 
\[
    \widehat{S} =  \sum_{k=0}^n A_k \circ (\mathbbm{1}+\Delta_k ), \quad \abs{\delta_{ijk}} \le  \varepsilon(n,b) \coloneqq \varepsilon^{\method{F}}(b)+\varepsilon^{\method{A}} (n / b)+\varepsilon^{\method{F}} (b) \varepsilon^{\method{A}} (n / b), \quad i,j = 1, \dots, d,
\]
and thus
\[
\norm{S-\widehat{S}} \le \varepsilon(n,b)\sum_{k=0}^n \norm{A_k}.
\]
\end{corollary}
In particular, if \textsc{AccurateSum} is calculated in double precision, i.e., $\varepsilon^{\method{A}}(n) = O(\ur^2)$, then the error bound is $\varepsilon(n,b) = \varepsilon^{\method{F}}(b)+O(\ur^2)$. Various options for the subroutines \textsc{AccurateSum} and \textsc{FastSum} are discussed in \cite{HighamBlanchard}.

\section{Summing matrix power series}\label{sec:pow}

Unlike the general matrix series considered in the last section, matrix power series admit more efficient algorithms. They  are also intimately connected to the study of matrix functions \cite{funcofmat}. The benefit of this connection goes both ways---the algorithms used to evaluate matrix functions, notably Pad\'e approximation and Schur--Partlett algorithm, may be adapted to implement the summation methods in Sections~\ref{sec:seq} and \ref{sec:func} numerically; the summation methods in Sections~\ref{sec:seq} and \ref{sec:func} may in turn be used to enhance these algorithms and to extend the domains of matrix functions.

For these purposes, the following basic definition of a matrix function \cite{funcofmat} suffices: If $X \in \mathbb{C}^{d \times d}$ and the power series
$f(z) = \sum_{k=0}^\infty a_k (z - z_0)^k$ converges in a neighborhood of $z_0 \in \mathbb{C}$, then 
\[
    f(X) \coloneqq \sum_{k=0}^\infty a_k (X - z_0 I)^k
\]
whenever the matrix power series on the right is summable in the conventional sense. By definition, the domain of $f$ is confined to
\[
\Omega \coloneqq
\biggl\{X \in \mathbb{C}^{d \times d} : \sum_{k=0}^\infty a_k (X - z_0 I)^k = S \text{ for some } S \in \mathbb{C}^{d \times d} \biggr\}.
\]
With hindsight from Sections~\ref{sec:seq} and \ref{sec:func}, we may define
\[
f(X) \overeq{R} \sum_{k=0}^\infty a_k  (X - z_0 I)^k
\]
with respect to any regular summation method $\mathsf{R}$, extending the domain of $f$ to a potentially larger domain 
\[
\Omega_\mathsf{R} \coloneqq \biggl\{X \in \mathbb{C}^{d \times d} : \sum_{k=0}^\infty a_k (X - z_0 I)^k \overeq{R} S \text{ for some } S \in \mathbb{C}^{d \times d} \biggr\} \supseteq \Omega.
\]
This portends a new vista in the study of matrix functions but any further exploration would take us too far afield.

We will instead limit our attention to the numerics and only to regular sequential summation methods in Section~\ref{sec:seq} as these work hand-in-glove with numerical algorithms for matrix functions. In this regard, there is no loss of generality to assume that $z_0 = 0$. As is the case in Section~\ref{sec:gen}, we begin by approximating $f(X)$ with its truncated Taylor Series $\sum_{k=0}^n a_k X^k$ for some $n \in \mathbb{N}$. But unlike the case of a general matrix series $\sum_{k=0}^\infty A_k$, working with a matrix power series $\sum_{k=0}^n a_k X^k$ permits us to ascertain $n$ in advance to achieve a desired $\varepsilon$-accuracy,
\[
\biggl\lVert f(X) -   \sum_{k=0}^n a_k X^k \biggr\rVert < \varepsilon
\]
as in \cite[Theorem~11.2.4]{GVL} or \cite[Corollary~2]{Roy} (see also \cite[Theorem~4.8]{funcofmat}).

We next see how we may add a summation method to the process. Let $\mathsf{R}$ be a regular sequential summation method \eqref{eq:sequential-transformation} such that $C_{n,k} \in \mathbb{C}^{d \times d}$ for all $n,k \in \mathbb{N}$ and $C_{n,k}=0$ for all $k > n$. The N\"orlund means (with Ces\`aro summation as a special case) in Section~\ref{sec:norlund} and Euler summation methods in Section~\ref{sec:euler} all meet this criterion. For any $A_k \in \mathbb{C}^{d \times d}$, $k \in \mathbb{N}$, and $S_n=\sum_{k=0}^n A_k$, observe that
\begin{equation}\label{eq:BSCA}
    \sum_{k=0}^n C_{n,k}S_k = \sum_{j=0}^n \biggl(\sum_{k=j}^n C_{n,k}\biggr) A_k.
\end{equation}
So for matrix power series the summation is characterized by the sums
\begin{equation}\label{eq:DB}
B_{n,k}\coloneqq \sum_{j=k}^{n} a_k C_{n,j}
\end{equation}
for $k \le n, \; k,n \in \mathbb{N}$.
Let $\varepsilon > 0$. If $\sum_{k=0}^\infty a_kX^k$ is $\mathsf{R}$-summable to $f(X)$, then for some $n \in \mathbb{N}$,
\begin{equation}\label{eq:trunc}
\biggl\lVert f(X) -   \sum_{k=0}^n B_{n,k} X^k \biggr\rVert   = \biggl \lVert  f(X) - \sum_{j=0}^{n}\biggl( \sum_{k=j}^{n} a_j C_{n,k} \biggr) X^j  \biggr \rVert  < \varepsilon.
\end{equation}
Using this, we will generalize Pad\'e approximation and the Schur--Parlett algorithm to work with  N\"orlund means and Euler summation. At this point, truncation error bounds like \cite[Theorem~11.2.4]{GVL} or \cite[Corollary~2]{Roy} that allow one to estimate $n$ from a given $\varepsilon$ are beyond our reach for \eqref{eq:trunc}. We will assume below, as we did in Section~\ref{sec:gen}, that $n$ is furnished as part of our inputs.

\subsection{Pad\'e approximation}\label{sec:pade}

This is one of the most powerful methods in matrix functions computations \cite[Section~4.4.2]{funcofmat}. The \texttt{expm} method in \textsc{Matlab}, which implements the scaling-and-squaring method to compute the matrix exponential  \cite{scaling-squaring}, is testament to one of the greatest wins\footnote{\url{https://blogs.mathworks.com/cleve/2024/01/25/nick-higham-1961-2024/}} of Pad\'e approximation. We will augment it with a regular sequential summation method $\mathsf{R}$.

An $(m,n)$-\emph{Pad\'e approximant} of $f(z) = \sum_{k=0}^\infty a_k z^k $ with respect to $\mathsf{R}$ is a rational function $[p/q](z)$ where $p(z) = \sum_{k=0}^m \beta_k z^k$, $q(z) = \sum_{k=0}^n \gamma_k z^k$, $\gamma_0 = 1$, and 
\begin{equation}\label{eq:pade}
    p(X)q(X)^{-1} = \sum_{k=0}^{m+n} B_{m+n,k} X^k,
\end{equation}
with $B_{m+n,0},B_{m+n,1},\dots,B_{m+n,m+n} \in \mathbb{C}^{d \times d}$ as defined in \eqref{eq:BSCA} and \eqref{eq:DB}. By this definition, the standard Pad\'e approximation in \cite[Section~4.4.2]{funcofmat} is then exactly the Pad\'e approximation with respect to conventional summation. 

Right multiplying $q(X)$ on both side of \eqref{eq:pade}, we get
\[
    \sum_{k=0}^m \beta_kX^k = \sum_{k=0}^{m+n} \biggl(\sum_{j=0}^k \gamma_{k-j} B_{m+n,j}\biggr) X^k.
\]
Since this holds for all $X \in \mathbb{C}^{d \times d}$, we may equate coefficients of $X^k$ on both sides to get
\begin{equation}\label{eq:pade-system}
    \sum_{j=0}^k \gamma_{k-j} B_{m+n,j} = \begin{cases}
        \beta_k I &\text{if } k = 0,\dots, m,\\
        0 &\text{if } k = m+1,\dots, n.
    \end{cases}
\end{equation}
For simplicity, we may choose a summation method $\mathsf{R}$ with $C_{n,k} = c_{n,k} I$ for some $c_{n,k} \in \mathbb{C}$ in \eqref{eq:BSCA} so that $B_{n,k} = b_{n,k}I$ for some $b_{n,k} \in \mathbb{C}$ in \eqref{eq:DB}. This simplification is not overly restrictive as it includes important methods such as Ces\`aro summation and Euler summation with $P = \rho I$ for $\rho >0$. The upside is that the coefficients of $p/q$ may be easily determined by solving for $\beta_k$ and $\gamma_k$ in a system of $m+n+1$ linear equations~\eqref{eq:pade-system}.
We summarize this in Algorithm~\ref{alg:pade}.

\begin{algorithm}[htb]
\caption{Pad\'e approximation with sequential summation} \label{alg:pade}
\begin{algorithmic}[1]
\Require $X \in \mathbb{C}^{d \times d}$, $m, n \in \mathbb{N}$, $a_k, c_{m+n,k} \in \mathbb{C}$ for $k=0,\dots, m+n$;
\For{$k = 0,\dots, m+n$}
    \State compute $b_{m+n,k} = a_k \sum_{j=k}^{m+n} c_{m+n,j} $;
\EndFor
\State solve the linear system \eqref{eq:pade-system} for $\beta_k$ for $k=0,\dots,m$ and  $\gamma_k$ for $k = 0,\dots, n$;
\State compute $P = \sum_{k=0}^m \beta_kX^k$ and $Q = \sum_{k=0}^n \gamma_kX^k$ with Algorithm~\ref{alg:horner}; \label{step:horner}
\Ensure $PQ^{-1}$.
\end{algorithmic}

\end{algorithm}

\begin{algorithm}
\caption{Horner's method}\label{alg:horner}
\begin{algorithmic}[1]
\Require $a_0,\dots,a_n\in \mathbb{R}$, $X \in \mathbb{R}^{d \times d}$;
\State initialize $P \leftarrow X$, $S \leftarrow a_0I + a_1X$;
\For{$ k = 2,\dotsc,n$}
    \State $P \leftarrow PX$;
    \State $S \leftarrow S + a_kP$;
\EndFor
\Ensure $S$.
\end{algorithmic}
\end{algorithm}

Algorithm~\ref{alg:horner} in Algorithm~\ref{alg:pade} (and also in Algorithm~\ref{alg:parlett} later) may be replaced by more sophisticated algorithms for evaluating matrix polynomials such as those in \cite{Paterson} or \cite[Section~4.6.4]{Knuth}, depending on whether one values stability or speed or yet other factors like parallelizability more.

Our approach in Algorithm~\ref{alg:pade} forms the Pad\'e approximant $Y = PQ^{-1}$ in the usual way: solving a linear system with multiple right-hand sides $Q^\tp Z = P^\tp$ and taking $Y = Z^\tp$. While there are alternative approaches via continued fractions and partial fractions, these are not necessarily stabler, as pointed out in  \cite[Section~4.4.3]{funcofmat}. More importantly, while a representation of a Pad\'e approximant in the form $PQ^{-1}$ is readily available through solving the linear system \eqref{eq:pade-system}, the same cannot be said of the other forms. Even for a function as standard as the matrix cosine function, it has been pointed out in \cite[p.~290]{funcofmat} that the Pad\'e approximant has no convenient continued fraction or partial fraction form.

We favor the $PQ^{-1}$ approach for yet a third reason: It is straightforward to incorporate the summation methods in Section~\ref{sec:seq}, as we did in \eqref{eq:pade}. We have added some experiments in Section~\ref{sec:pade-exp} to show how Algorithm~\ref{alg:pade} works in conjunction with Ces\`aro and Euler summations, allowing us to sum a power series in regions far outside its usual range of convergence.

\subsection{Schur--Parlett algorithm}

The `Schur' part of this algorithm is routine: To evaluate $f(X) = \sum_{k=0}^\infty a_k X^k$ for $X \in \mathbb{C}^{d \times d}$, a Schur decomposition $X = QRQ^*$ with unitary $Q \in \mathbb{C}^{d \times d}$ and upper triangular $R \in \mathbb{C}^{d \times d}$ yields $f(X) = Qf(R)Q^*$, thus reducing the problem to computing $f(R)$.

The `Parlett' part of this algorithm is where the innovation lies: partition $R \in \mathbb{C}^{d\times d}$ into an $r \times r$ block matrix $R=(R_{ij})$, $i,j \in 1,\dots, r$, with square diagonal blocks $R_{ii}$, $i = 1, \dots, r$. Parlett \cite{Parlett} observed  that the matrix $F = f(R)$ commutes with $R$; has the same block structure $F = (F_{ij})$; and upon evaluating the diagonal blocks $F_{ii} = f(R_{ii})$, the superdiagonal blocks can be obtained from a system of Sylvester equations
\begin{equation}\label{eq:Parlett}
    R_{ii}F_{ij}-F_{ij}R_{jj} = F_{ii}R_{ij} - R_{ij}F_{jj} + \sum_{k=i+1}^{j-1}(F_{ik}R_{kj}-R_{ik}F_{kj}), \quad 1\leq i< j\leq d.
\end{equation}
The system \eqref{eq:Parlett} is nonsingular if and only if $R_{ii}$ and $R_{jj}$ have no eigenvalue in common \cite[Section~D.14]{funcofmat}. Fortunately this may be guaranteed   \cite{Higham-Schur--Parlett}  by further transforming $R$ into an identically partitioned upper triangular matrix $T= VRV^*$ with some unitary  $V \in \mathbb{C}^{d \times d}$  such that for a fixed $\delta > 0$,
\begin{enumerate}[\normalfont(i)]
    \item $\min\{\lvert \lambda_i - \lambda_j \rvert: \lambda_i \in \lambda(T_{ii}), \lambda_j \in \lambda(T_{jj}), i \neq j\} > \delta$; \label{cond:reorder1}
    \item if the block $T_{ii}$ has dimension greater than $1$, then every $\lambda \in \lambda(T_{ii})$ has a corresponding $\mu \in \lambda(T_{ii})$ with $\mu \neq \lambda$ and $\lvert \lambda - \mu \rvert \leq \delta$. \label{cond:reorder2}
\end{enumerate}
Essentially \ref{cond:reorder1} says that between-block eigenvalues are well separated; and \ref{cond:reorder2} says that within-block eigenvalues are closely clustered.
Since $X = (QV^*) T (QV^*)^*$ and $F = (QV^*) f(T) (QV^*)^*$, we may use $T$ in place of $R$. This additional transformation from $R$ to $T$ carries other numerical advantages \cite[Section~9.3]{funcofmat} in the solution of \eqref{eq:Parlett}.

We augment the Schur--Parlett algorithm with a regular sequential summation method $\mathsf{R}$ by computing the diagonal blocks $F_{ii}$ as $\mathsf{R}$-sums, i.e., 
\begin{equation}\label{eq:FR}
F_{ii} = f(R_{ii}) \approx \sum_{k=0}^n B_{n,k} R_{ii}^k = \sum_{k=0}^n \biggl(\sum_{j=k}^n a_k C_{n,j} \biggr) R_{ii}^k, \quad i = 1, \dots, r,
\end{equation}
where $B_{n,k}$, $C_{n,j} \in \mathbb{C}^{d \times d}$ are as defined in \eqref{eq:BSCA} and \eqref{eq:DB}. Note that the diagonal blocks $R_{ii}$'s in \eqref{eq:FR} would in general have different dimensions for different $i$, which means that the matrices $B_{n,k}$, $C_{n,j}$ would need to have dimensions the same as $R_{ii}$'s and therefore chosen differently for each $i$. While there is no reason why we cannot do this we provide a simple workaround---we just set $C_{n,k} = c_{n,k}I$ for some $c_{n,k} \in \mathbb{C}$ as we did in Section~\ref{sec:pade}.

This simplification in turn constraints us to use $C_{n,k} = c_{n,k} I$ for some $c_{n,k} \in \mathbb{C}$ in \eqref{eq:BSCA} but the result is both dimension-independent and computationally efficient as it only requires computing scalar coefficients $\sum_{k=j}^n a_j c_{n,k}$ as opposed to matrix coefficients. We summarize this in Algorithm~\ref{alg:parlett}.

\begin{algorithm}
\caption{Schur--Parlett algorithm with sequential summation} \label{alg:parlett}
\begin{algorithmic}[1]
\Require $X \in \mathbb{C}^{d \times d}$, $n \in \mathbb{N}$, $a_k, c_{n,k} \in \mathbb{C}$ for $k=0,\dots, n$;
\State compute the Schur decomposition $X = QRQ^*$;
\State compute $T= VRV^*$ with block partition satisfying Conditions~\ref{cond:reorder1} and \ref{cond:reorder2};
\State $R \leftarrow T$; 
\State $Q \leftarrow QV^*$;
\For{$i = 1,\dots, r$}
    \State compute $F_{ii} = \sum_{k=0}^n(\sum_{j=k}^n a_k c_{n,j})R_{ii}^k$ by Algorithm~\ref{alg:horner};
\EndFor
\For{$j = 2,\dots, r$}
    \For{$i = j-1,j-2,\dots,1$}
        \State solve for $F_{ij}$ in the Sylvester equation \eqref{eq:Parlett};
    \EndFor
\EndFor
\Ensure $QFQ^*$.
\end{algorithmic}
\end{algorithm}
Compared to directly summing \eqref{eq:trunc}, Algorithm~\ref{alg:parlett} dramatically improves computational time, as we will see in Section~\ref{sec:neu}.

\section{Numerical experiments} \label{sec:numer}

We will present numerical experiments to illustrate the use of the summation methods in Sections~\ref{sec:seq} and \ref{sec:func} in conjunction with the numerical algorithms in Sections~\ref{sec:gen} and \ref{sec:pow}. Because of the large number of possible combinations, it is not possible to be exhaustive although we try to present a diverse selection. Our experiments will see four types of matrix series (Taylor, Fourier, Dirichlet, Hadamard); two sequential summations (Ces\`aro and Euler), two functional summations (Borel and Lambert); and three numerical algorithms (Schur--Parlett algorithm, recursive, and compensated summations). Each experiment is designed to showcase a different utility of these methods and algorithms.
\begin{center}
\begin{tabular}{ll}
Section~\ref{sec:gibbs}: &  using Ces\`aro sums to alleviate Gibbs phenomenon in matrix Fourier series; \\
Section~\ref{sec:neu}: &  using Euler and strong Borel sums to extend matrix Taylor series; \\
Section~\ref{sec:acc}: &  using Euler sums for high accuracy evaluation of matrix functions; \\
Section~\ref{sec:pade-exp}: & using Ces\`aro and Euler summations in Pad\'e approximations; \\
Section~\ref{sec:dirichlet}: & using Lambert sums to investigate matrix Dirichlet series;\\
Section~\ref{sec:comp_numerical}: &  using compensated summation for accurate evaluation of Hadamard power series.
\end{tabular}
\end{center}
All experiments are performed with \textsc{Matlab} R2023a in double precision ($\ur = 2^{-52} \approx 2.22 \times 10^{-16}$) arithmetic unless noted otherwise. Plots presented in log scale would all be in base $10$. All codes have been made available at
\begin{quote}
\url{https://github.com/thomasw15/Summing-Divergent-Matrix-Series}.
\end{quote}

\subsection{Avoiding Gibbs phenomenon with Ces\`aro summation}\label{sec:gibbs}

When one attempts to approximate a discontinuous function with its Fourier series, the Fourier approximation inevitably overshoots near a point of discontinuity---the notorious Gibbs phenomenon. The canonical example is given by the square wave function $f : \mathbb{R} \to \mathbb{R}$,
\begin{equation}\label{eq:sq}
    f(x) = \begin{cases}
        1 & 2k\pi\leq x<(2k+1)\pi, \; k\in \mathbb{Z},\\
        -1 & (2k-1)\pi \le x < 2k\pi, \; k\in \mathbb{Z}.
    \end{cases}
\end{equation}
Attempting to approximate $f$ by its $1000$-term Fourier series
\begin{equation}\label{eq:1000}
   f_{1000}(x) = \sum_{k=1}^{1000} \frac{2}{\pi k}\bigl( 1-(-1)^k \bigr) \sin(kx)
\end{equation}
produces the blue curve in Figure~\ref{fig:1dgibbs}, which prominently overshoots near $0$ and $\pm \pi$, the points of discontinuity. A close-up look at the absolute value of the errors in a neighborhood of $x = 0$ further reveals the wildly oscillatory nature of $f_{1000}$, shown in the blue curve in Figure~\ref{fig:1dgibbsabs}.
\begin{figure}[htb]
  \begin{subfigure}{0.49\textwidth}
    \includegraphics[trim={10em 52ex 11em 55ex},clip,width=\textwidth]{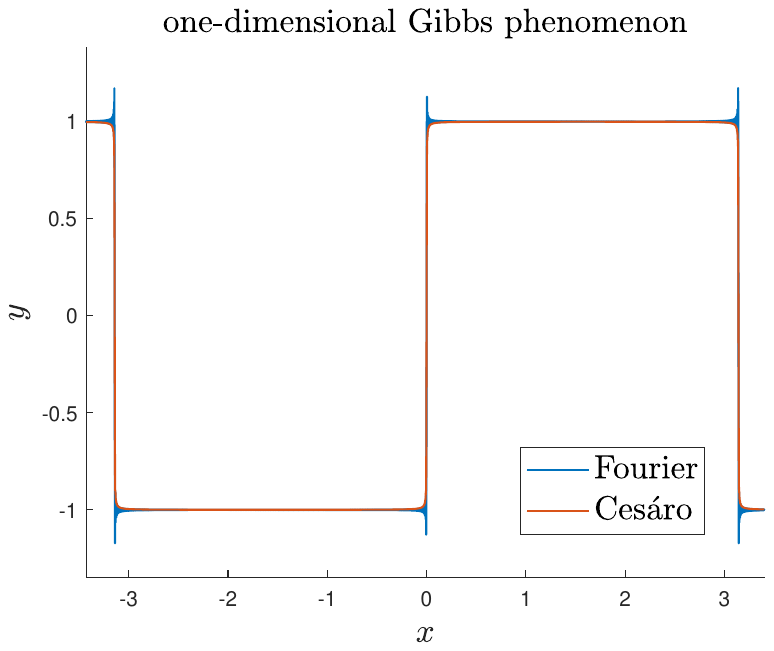}
    \caption{Approximating square wave.} \label{fig:1dgibbs}
  \end{subfigure}%
  \hspace*{\fill} 
  \begin{subfigure}{0.49\textwidth}
    \includegraphics[trim={9.5em 52ex 11em 55ex},clip,width=\textwidth]{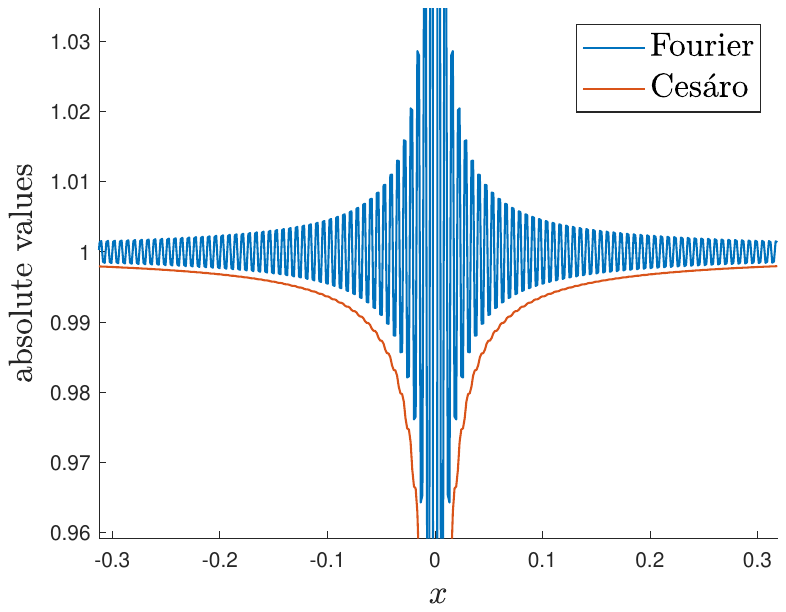}
    \caption{Absolute value of errors near $0$.} \label{fig:1dgibbsabs}
  \end{subfigure}%
\caption{Gibbs phenomena in Fourier series corrected with Ces\`aro sum.} \label{fig:1gibbs}
\end{figure}

While the Gibbs phenomenon may be ameliorated with ad hoc recipes like the Lanczos factor \cite{Acton}, a superior remedy would be to use Ces\`aro summation. The same data in \eqref{eq:1000} yields the Ces\`aro partial sum 
\begin{equation}\label{eq:1000cesaro}
    \sigma_{1000}(x) = \frac{1}{1000}\sum_{n=1}^{999} \sum_{k=1}^n \frac{2}{\pi k}\bigl( 1-(-1)^k \bigr) \sin(kx),
\end{equation}
which nearly eliminates the wild oscillations completely, as shown in the red curves in Figures~\ref{fig:1dgibbs} and \ref{fig:1dgibbsabs}. While this is well known, the following matrix version is new, as far as we know.

Consider the following matrix Fourier series and its corresponding Ces\`aro sum:
\begin{equation}\label{eq:100}
\begin{aligned}
    F_{100}(X(t)) &= \sum_{k=1}^{100} \frac{2}{\pi k}\bigl(1-(-1)^k \bigr) \sin(kX(t)), \\
  \Sigma_{100}(X(t)) &= \frac{1}{99}\sum_{n=1}^{99}\sum_{k=1}^{n} \frac{2}{\pi k}\bigl(1-(-1)^k \bigr) \sin(kX(t))
\end{aligned}
\end{equation}
where $X : \mathbb{R} \to \mathbb{R}^{1000 \times 1000}$ is a continuous matrix-valued function with $\lambda(X(0)) = 0$. Note that each summand involves a matrix sine function \cite[Chapter~12]{funcofmat}, which we compute with the \textsc{Matlab} function \texttt{funm(X,@sin)}.

In a neighborhood of $x = 0$ the square wave function \eqref{eq:sq} is identical to the sign function, i.e., $\sign(x) = 1$ if $x \ge 0$ and $-1$ if $x < 0$. It is therefore conceivable that the same would hold for matrix functions and that the sums in \eqref{eq:100} should approximate the matrix sign function $\sign(X)$ \cite[Chapter~5]{funcofmat} in a neighborhood of $X = 0$.  Surprisingly this is only true if $X$ is diagonalizable and false otherwise, a fact we discovered through the following numerical experiments.

Consider the obviously nondiagonalizable matrix $X(t) = \diag(J_1(t), \dots, J_{100}(t))$ where $J_i: \mathbb{R} \to \mathbb{R}^{10 \times 10}$ is given by
\[
J_i(t) = \begin{bmatrix}
    t \lambda_i &1 & &\\
    &\ddots &\ddots &\\
    & &\ddots &1\\
    & & &t \lambda_i
\end{bmatrix}.
\]
Using the compensated summation in Algorithm~\ref{alg:comp}, we compute the two sums in \eqref{eq:100} and compare their norms with that of the matrix sign function in Figure~\ref{fig:jordan}. 
\begin{figure}[htb]
\includegraphics[trim={10em 51ex 12.3em 52ex},clip,width=0.5\textwidth]{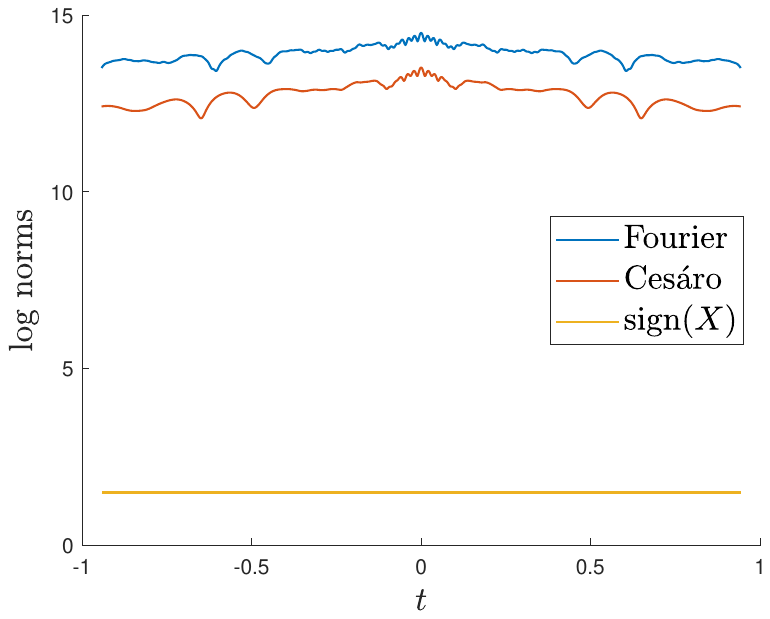}
    \caption{Failure to approximate matrix sign function for nondiagonalizable $X$.} \label{fig:jordan}
\end{figure}

The result shows that the sums in \eqref{eq:100} bear no resemblance to the matrix sign function---both $\lVert F_{100}(X(t))\rVert$ and $\lVert \Sigma_{100}(X(t)) \rVert$ are orders of magnitude away from $\lVert \sign(X(t)) \rVert$. With hindsight, the reason is clear, as the sums in \eqref{eq:100} will always involve the superdiagonal of $1$'s, whereas these play no role in the matrix sign function. While we have chosen $X(t)$ above to accentuate this effect, the argument holds true as long as there is a single Jordan block of size at least $2 \times 2$, i.e., as long as the matrix is not diagonalizable.

On the other hand, the sums in $\eqref{eq:100}$ give a fair approximation of $\sign(X)$ for a diagonalizable matrix $X$ and, as expected, we see prominent Gibbs phenomenon in $F_{100}(X)$ that is alleviated in $\Sigma_{100}(X)$. We will give a symmetric and a nonsymmetric example by randomly generating $\lambda_1,\dots,\lambda_{1000} \in \mathbb{R}$,  orthogonal $Q$ and nonsingular tridiagonal $T \in \mathbb{R}^{1000 \times 1000}$, and defining
\[
Y(t) = Q \diag(t \lambda_1,\dots, t \lambda_{1000})Q^\tp, \qquad Z(t) = T\diag(t \lambda_1,\dots, t \lambda_{1000}) T^{-1}.
\]
We approximate the square wave function with the matrix Fourier series and its Ces\`aro sum in \eqref{eq:100}, with $Y(t)$ and $Z(t)$ in place of $X(t)$, relying again on Algorithm~\ref{alg:comp} to compute the sums.

\begin{figure}[htb]
\begin{subfigure}{0.49\textwidth}
\includegraphics[trim={10em 51ex 12em 55ex},clip,width=\textwidth]{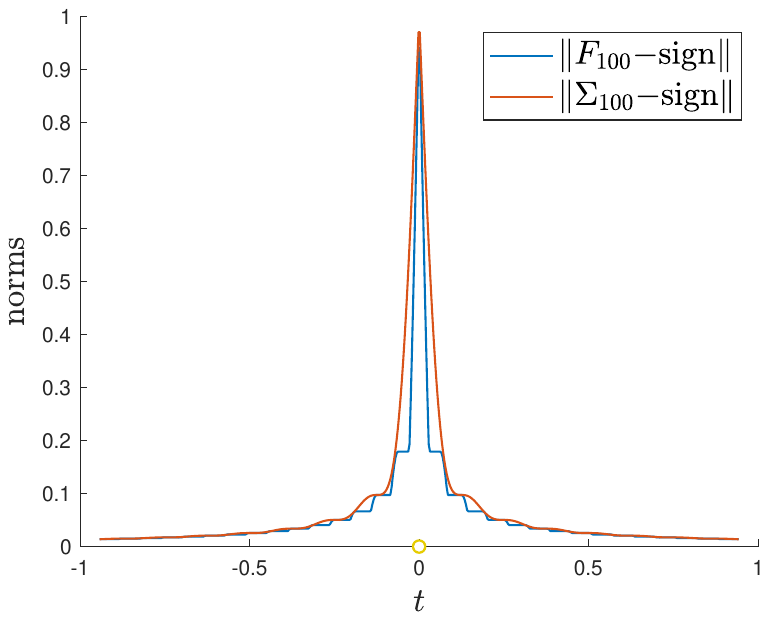}
\caption{symmetric matrices} \label{fig:ortho_norm}
\end{subfigure}%
\begin{subfigure}{0.49\textwidth}
\includegraphics[trim={10em 51ex 12em 55ex},clip,width=\textwidth]{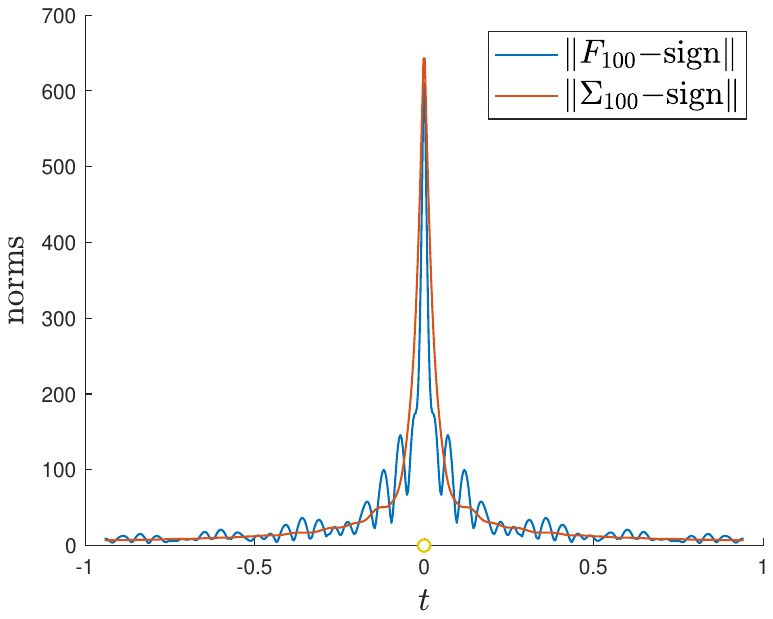}
\caption{nonsymmetric diagonalizable matrices} \label{fig:tri_norm}
\end{subfigure}%
\caption{Matrix sign function approximated by matrix Fourier and Ces\`aro sums.} \label{fig:matrix_Gibbs_norm}
\end{figure}

Outside $t=0$, where the matrix sign function is undefined, both $F_{100}$ and $\Sigma_{100}$ provide fair approximations as quantified by $\lVert F_{100}(Y(t)) - \sign(Y(t))\rVert$ and $\lVert \Sigma_{100}(Y(t)) - \sign(Y(t)) \rVert$ in Figure~\ref{fig:matrix_Gibbs_norm}. We expect the approximation errors to further decrease as the number of terms increases beyond $100$. For comparison the more accurate approximations in Figure~\ref{fig:1dgibbs} for the scalar series took a $1000$-term approximation, which is beyond our reach here for $1000 \times 1000$ matrix series.

\begin{figure}[htb]
\begin{subfigure}{0.49\textwidth}
\includegraphics[trim={10em 51ex 12em 55ex},clip,width=\textwidth]{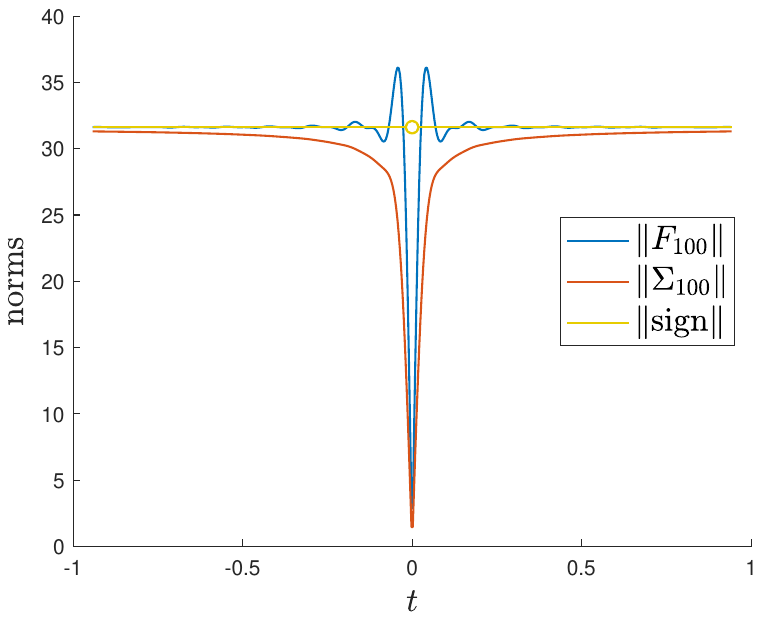}
\caption{symmetric matrices} \label{fig:ortho}
\end{subfigure}%
\begin{subfigure}{0.49\textwidth}
\includegraphics[trim={10em 51ex 12em 55ex},clip,width=\textwidth]{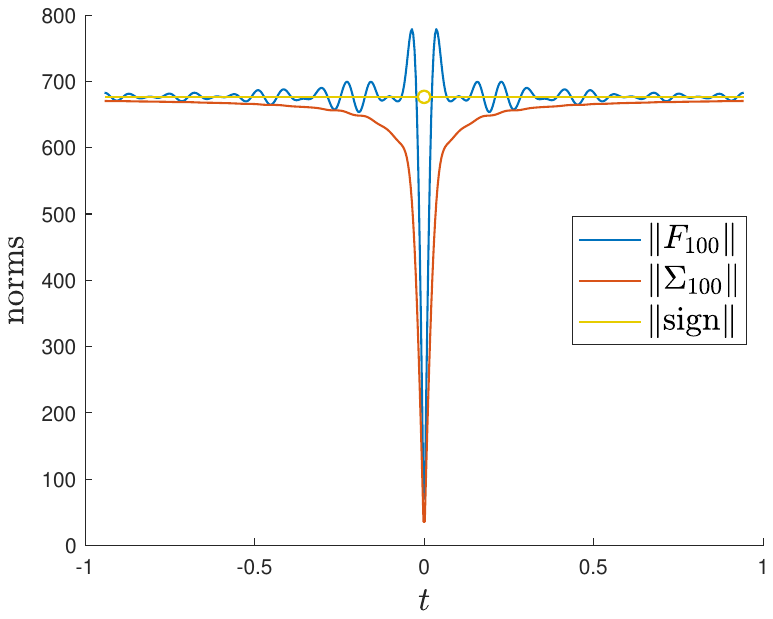}
\caption{nonsymmetric diagonalizable matrices} \label{fig:tri}
\end{subfigure}%
\caption{Gibbs phenomena in matrix Fourier series corrected with Ces\`aro sum.} \label{fig:matrix_Gibbs}
\end{figure}

In a neighborhood of $t = 0$, we see the unmistakable mark of Gibbs phenomenon in $F_{100}$, reflected in the norms of $\lVert F_{100}(Y(t)) \rVert$ and $\lVert F_{100}(Z(t)) \rVert$, the blue curves in Figures~\ref{fig:ortho} and \ref{fig:tri} respectively. The oscillatory behavior vanishes when we instead look at the corresponding Ces\`aro sums $\Sigma_{100}$, whose norms are given by the red curves in Figure~\ref{fig:matrix_Gibbs}. This indicates that for diagonalizable matrices, Ces\`aro summation is a remedy for Gibbs phenomenon in matrix Fourier series.

\subsection{Accurate summation with Euler method and strong Borel method}\label{sec:neu}

These experiments accomplish two goals. We first verify numerically that the Euler and strong Borel methods indeed extend the domain of Neumann series beyond $\mathbb{D}$, which we demonstrated analytically in Corollary~\ref{cor:euneu} and Proposition~\ref{prop:borneu}. The experiments for Euler methods are also used to show that the Schur--Parlett algorithm  for Euler summation, i.e., Algorithm~\ref{alg:parlett} with $c_{n,k} = \binom{n+1}{k+1} \rho^{n-k}(1+\rho)^{-n-1}$, is less accurate but dramatically faster than directly computing with Algorithm~\ref{alg:comp}.

We generate fifty matrices $X\in \mathbb{C}^{1000\times 1000}$ such that $\lambda(X)\subseteq \{ z \in \mathbb{C} : \lvert z + \rho \rvert < 1 + \rho \}$ for $\rho =10^4$ and $	\lambda(X) \not\subseteq \mathbb{D}$. Note that the Neumann series $\sum_{k=0}^\infty X^k$ for such matrices will not be conventionally summable. Our goal is to verify that Euler method and strong Borel method will however yield the expected $(I - X)^{-1}$ numerically.  For Euler method, we compute the truncated $(\mathsf{E},\rho)$ sum as defined in \eqref{eq:eq_neumann},
\[
    \widehat{S}_{\doublemethod{E}{\rho}} \coloneqq \sum_{k=0}^{10000} \E_k^{\rho}(X_\bullet)
\]
where $X_\bullet = (X^k)_{k=0}^\infty$,  first with compensated summation and then with the Schur--Parlett algorithm in single precision $(\ur = 2^{-23} \approx 1.19 \times 10^{-7})$.
For the strong Borel method, we use the \textsc{Matlab} function \texttt{integral} with tolerance level $10^{-12}$ to compute the Borel sum $\widehat{S}_{\method{SB}}$ as in \eqref{eq:SB_Neumann} in single precision.

We plot the forward errors $\norm{\widehat{S} - (I-X)^{-1}}$ in Figure~\ref{fig:neuf} and the backward errors $\norm{\widehat{S}(I-X)-I}$ in Figure~\ref{fig:neub}, where $\widehat{S}$ is either $\widehat{S}_{\doublemethod{E}{\rho}}$ or $\widehat{S}_{\method{SB}}$. The near zero errors are strong numerical evidence that both Euler and strong Borel methods analytically extend the Neumann series to $(I-X)^{-1}$, which we of course know is true by virtue of Corollary~\ref{cor:euneu} and Proposition~\ref{prop:borneu}. 

\begin{figure}[hbt]
  \begin{subfigure}{0.49\textwidth}
    \includegraphics[trim={9.5em 51ex 12em 55ex},clip,width=\textwidth]{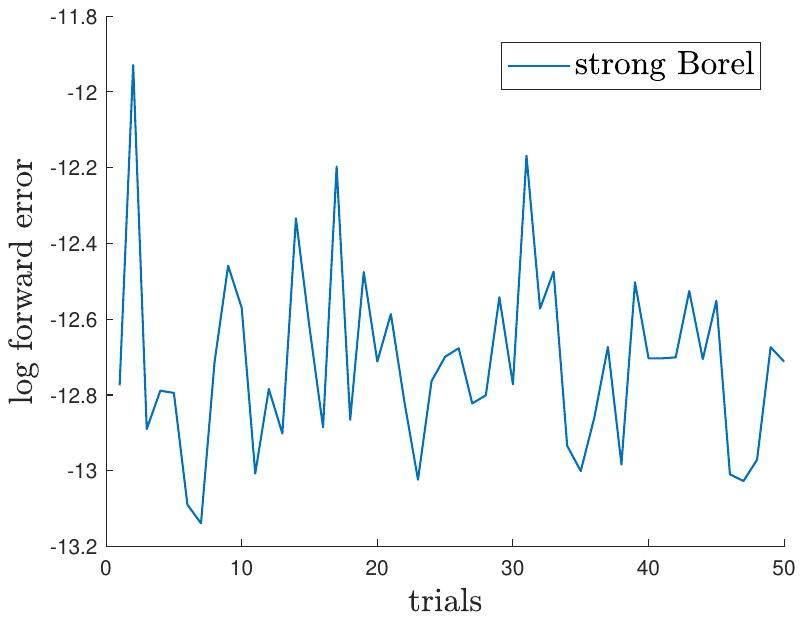}
    \caption{strong Borel method} \label{fig:borelf}
  \end{subfigure}
  \begin{subfigure}{0.49\textwidth}
    \includegraphics[trim={9.5em 51ex 12em 55ex},clip,width=\textwidth]{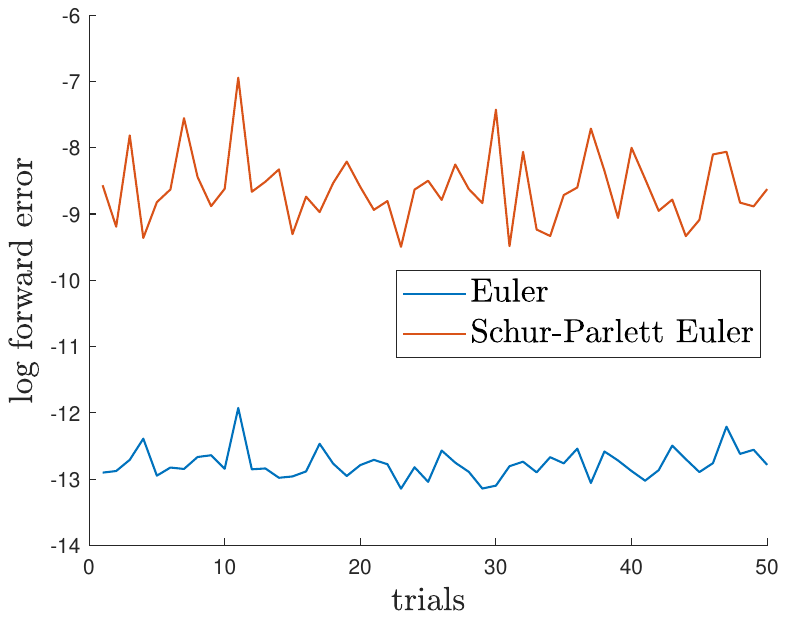}
    \caption{Euler method} \label{fig:eqf}
  \end{subfigure}%
\caption{Log forward errors of strong Borel and Euler summations.} \label{fig:neuf}
\end{figure}

Observant readers might have noticed an issue here. We do not really have $(I-X)^{-1}$ exactly but only the output of the \texttt{inv} function in \textsc{Matlab}, which is also subjected to floating point and approximation errors. Indeed our `forward errors' here are simply a measure of deviation from $\widehat{S}_{\texttt{inv}}$, the result of \texttt{inv} applied to $I-X$, computed in double precision. The backward errors $\norm{\widehat{S}(I-X)-I}$  for $\widehat{S}_{\texttt{inv}}$,  $\widehat{S}_{\doublemethod{E}{\rho}}$, $\widehat{S}_{\method{S B}}$, computed in double precision, provide a more equitable comparison and therein lies a surprise---when computed with compensated summation, $\widehat{S}_{\doublemethod{E}{\rho}}$, the result of Euler method, is more accurate than  $\widehat{S}_{\texttt{inv}}$, the result of \textsc{Matlab}'s \texttt{inv}, as is evident in Figure~\ref{fig:eqb}.

\begin{figure}[hbt]
  \begin{subfigure}{0.49\textwidth}
    \includegraphics[trim={9.5em 51ex 12em 55ex},clip,width=\textwidth]{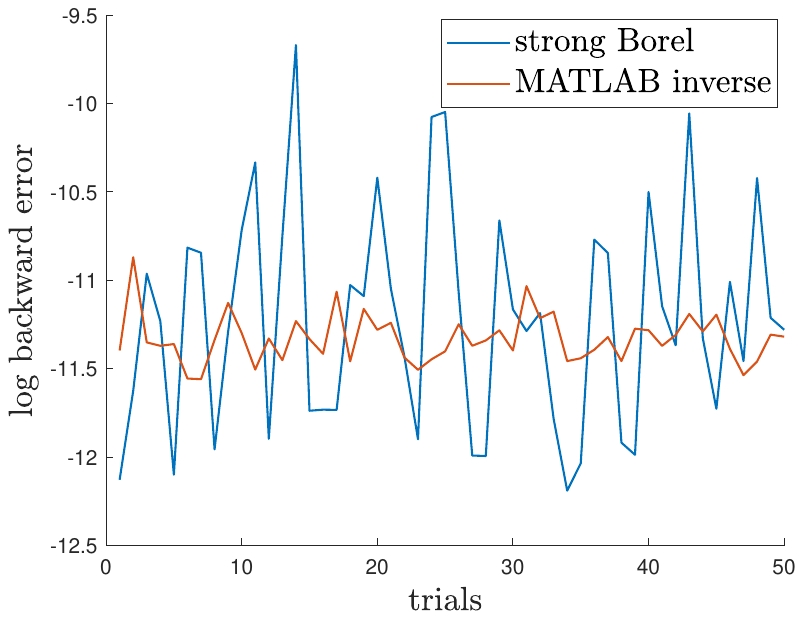}
    \caption{strong Borel method} \label{fig:borelb}
  \end{subfigure}
  \begin{subfigure}{0.49\textwidth}
    \includegraphics[trim={9.5em 51ex 12em 55ex},clip,width=\textwidth]{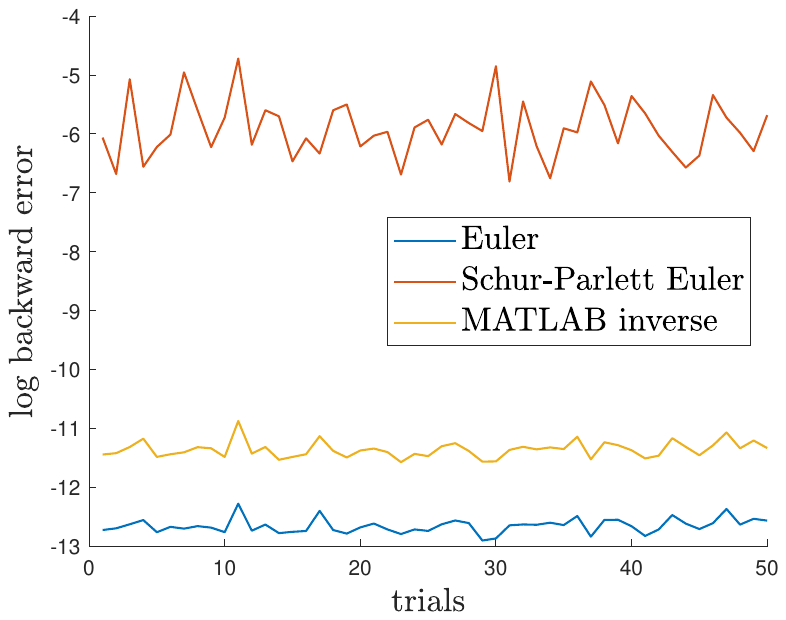}
    \caption{Euler method} \label{fig:eqb}
  \end{subfigure}%
\caption{Log backward errors of strong Borel and Euler summations.} \label{fig:neub}
\end{figure}

As both forward and backward errors in Figures~\ref{fig:eqf} and \ref{fig:eqb} reveal, the Schur--Parlett algorithm gives less accurate results for  Euler sums than  compensated summation. However, a comparison of their running times in Figure~\ref{fig:SchurParlett} shows that the former is significantly faster.

\begin{figure}[htb]
    \includegraphics[trim={9.5em 51ex 12em 55ex},clip,width=0.5\textwidth]{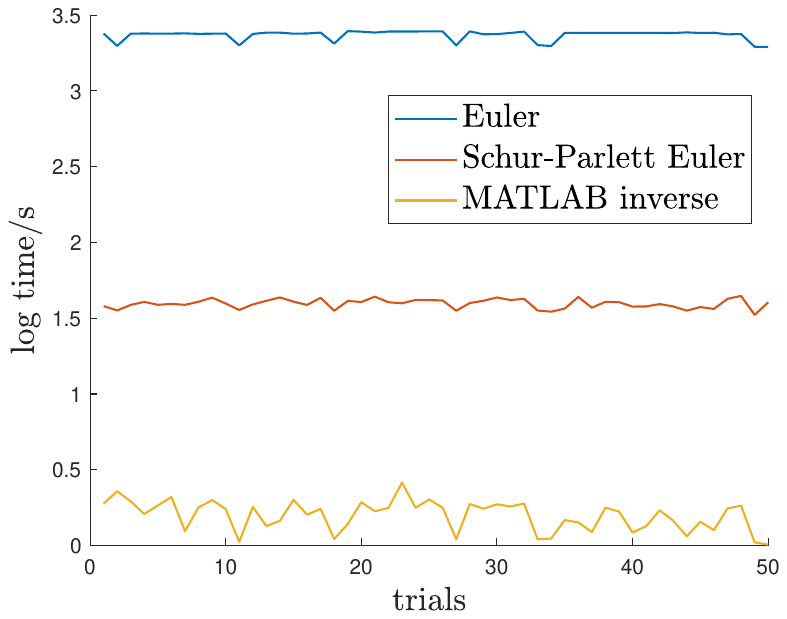}
    \caption{Run time comparison of compensated summation and Schur--Parlett algorithm on Euler sums.} \label{fig:SchurParlett}
  \end{figure}%

\subsection{High accuracy sums with Euler methods}\label{sec:acc}

The surprising accuracy of Euler method computed with compensated summation uncovered in Section~\ref{sec:neu} deserves a more careful look. Here we will examine how the value of $\rho$ impacts its accuracy.

We generate twenty bidiagonal matrices $X\in \mathbb{R}^{1000 \times 1000}$ whose diagonal entries are negative with probability  $\alpha \in \{ 0, 0.01,\dots,0.99,1\}$. These matrices are generally not diagonalizable but we may readily prescribe their eigenvalue distribution. Again we will use the Neumann series $\sum_{k=0}^\infty X^k$, whose value $S = (I-X)^{-1}$ is known, as our test case. We approximate it with a $100$-term truncated Taylor series and a $100$-term Euler sum
\[
    \widehat{S}\coloneqq \sum_{k=0}^{100} X^k \quad \text{and} \quad \widehat{S}_{\doublemethod{E}{\rho}} \coloneqq \sum_{k=0}^{100} \E_k^\rho (X_\bullet),
\]
with $\rho \in \{ 1, 1/2, 1/4 \}$, using compensated summation in Algorithm~\ref{alg:comp} to compute these sums. For a bidiagonal $X$ we know $S = (I-X)^{-1}$ exactly in closed form and do not need to rely on \textsc{Matlab}'s \texttt{inv}, we may compute the forward errors $\norm{\widehat{S}_{\doublemethod{E}{\rho}} - S}$ and $\norm{\widehat{S} - S}$. The logarithm of these values averaged over the twenty trials are plotted against $\alpha$ in Figure~\ref{fig:accelerate}.
\begin{figure}[htb]
\includegraphics[trim={9em 51ex 9em 55ex},clip,width=0.55\textwidth]{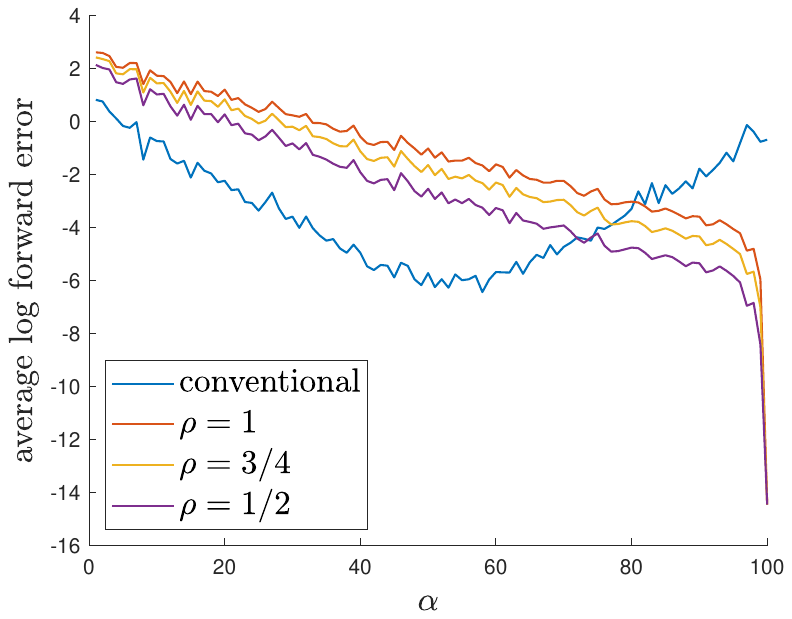}
    \caption{Log errors from Euler methods.}
    \label{fig:accelerate}
\end{figure}

We highlight two observations. Firstly, the downward trend of the curves for Euler method with increasing $\alpha$ shows that when the eigenvalues are predominantly negative, a truncated Euler sum gives a much higher level of accuracy with $100$ terms than a truncated Taylor series with the same number of terms. This implies that Euler sums converge much faster than Taylor series. Secondly, when using Euler summation, smaller values of $\rho$ lead to faster convergence than larger ones.

\subsection{Extending summability range with Ces\`aro and Euler summations}\label{sec:pade-exp}

For any $\alpha \in \mathbb{R}$, the binomial series on the left
\begin{equation}\label{eq:binom}
\sum_{k=0}^\infty \binom{\alpha}{k}X^k = (I + X)^\alpha
\end{equation}
converges to the value on the right in the conventional sense whenever $\lambda(X) \subseteq (-1,1)$. Here $\binom{\alpha}{k} \coloneqq \alpha(\alpha-1)\cdots(\alpha-k+1)/k!$ is the binomial coefficient, defined for any $\alpha \in \mathbb{R}$. This provides a good test case as the series on the left is an infinite series for any non-integer  $\alpha$ but sums to the closed-form expression on the right conventionally if every eigenvalue $\lambda$ of $X$ satisfies $\lvert \lambda \rvert < 1$. However if there is any eigenvalue with $\lvert \lambda \rvert > 1$, then the series on the left necessarily diverges in the conventional sense.  Our experiment will show that this range  can be vastly extended.

We compute the $(m,n)$-Pad\'e--Ces\`aro approximants, i.e., the $B_{m+n,k}$'s in \eqref{eq:pade} are given by
\[
B_{m+n,k} = \binom{\alpha}{k} \frac{m+n+1-k}{m+n+1}I.
\]
We also compute the $(m,n)$-Pad\'e--Euler approximants, i.e., the $B_{m+n,k}$'s in \eqref{eq:pade} are from the $P$-Euler transform in \eqref{eq:Aq} with $P = \rho I$:
\[
B_{m+n,k} = \binom{\alpha}{k}\sum_{j=k}^{m+n} \binom{m+n+1}{j+1}\frac{\rho^{m+n-j}}{(1+\rho )^{m+n+1}}I.
\]
We set $m = n$ and denote these approximants by $\widehat{S}_{\mathsf{C},n}$ and $\widehat{S}_{\mathsf{E},n}$ respectively. We fix $\rho = 100$  and let $\alpha$ and $n$ run over
\[
\alpha = \pm \frac14, \; \pm \frac12, \; \pm \frac34, \; \pm \frac35, \; \pm \frac{4}{7}; \qquad n = 1,2,\dots,20.
\]
For each value of $\alpha$ and $n$, we repeat our experiments for ten matrices $X \in \mathbb{R}^{10 \times 10}$, generated as $X = QRQ^\tp$ with a random orthogonal matrix $Q$ and a random upper triangular $R$ with diagonal entries randomly chosen in $[75, 150]$. In other words,  the spectrum  $\lambda(X) \subseteq [75,150]$, way beyond the range of convergence $(-1,1)$ of the binomial series.

\begin{figure}[h]
  \begin{subfigure}{0.49\textwidth}
    \includegraphics[trim={9.5em 51ex 12em 55ex},clip,width=\textwidth]{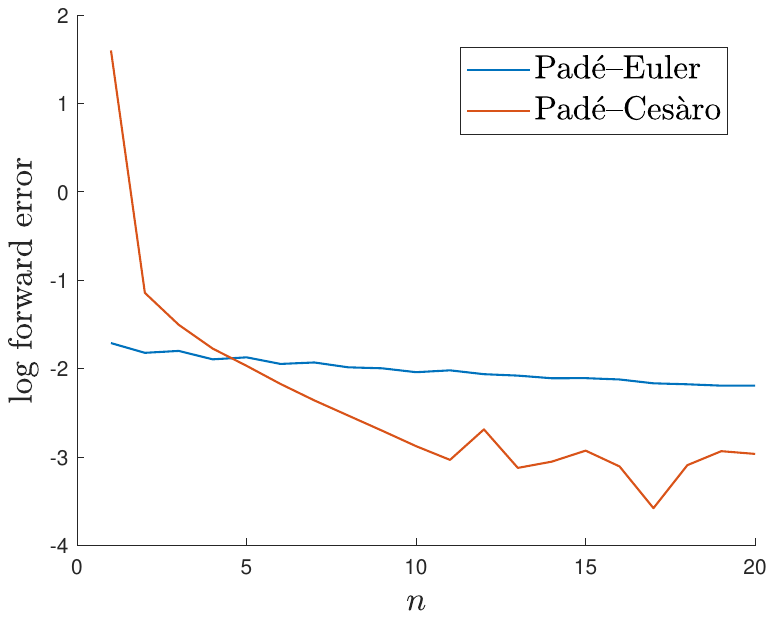}
    \caption{increasing order of approximants} \label{fig:pade-range}
  \end{subfigure}
  \begin{subfigure}{0.49\textwidth}
    \includegraphics[trim={9.5em 51ex 11em 55ex},clip,width=\textwidth]{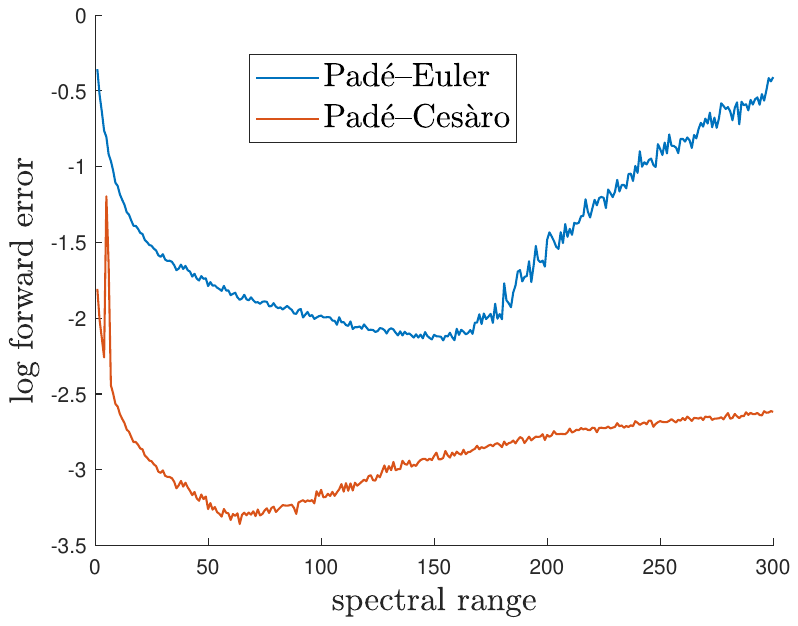}
    \caption{increasing spectral range of $X$} \label{fig:pade-spectral}
  \end{subfigure}%
\caption{\blue{Log forward errors of Pad\'e--Ces\`aro and Pad\'e--Euler approximants.}} \label{fig:pade}
\end{figure}

There is not much variation across different values of $\alpha$ and so we will present a typical case $\alpha = -3/4$.
We will treat $(I+X)^{-3/4}$ computed using \textsc{Matlab}'s \texttt{mpower} function as the true value. With this, we obtain forward errors
$\lVert \widehat{S}_{\mathsf{C},m} - (I+X)^{-3/4}\rVert_{\infty,1}$ and $\lVert \widehat{S}_{\mathsf{E},m} - (I+X)^{-3/4}\rVert_{\infty,1}$;
note that we are computing absolute errors in the $(\infty,1)$-norm $\lVert X \rVert_{\infty,1} = \max_{i,j=1,\dots,n}\; \lvert x_{ij} \rvert$ in order to show the number of correct decimal digits. We plot their average values on a log scale in Figure~\ref{fig:pade-range}. As is clear from the plots, both Pad\'e approximants converge to the right-hand side of \eqref{eq:binom}, despite being in a spectral range that falls far outside $(-1,1)$.

To see how far we can go before these approximants begin to show signs of divergence, we generate $X$ as described above but now with $\lambda(X) \subseteq [\frac12 r, r]$ for $r$ going up to $300$. We show the forward error for the $(15,15)$-Pad\'e--Ces\`aro and Pad\'e--Euler approximants on a log scale in Figure~\ref{fig:pade-spectral}. While Pad\'e--Euler  starts diverging when $\lambda(X) \subseteq [80, 160]$, Pad\'e--Ces\`aro is remarkably resilient, providing reasonably values even when $\lambda(X) \subseteq [150, 300]$.

\subsection{Matrix Dirichlet series with Lambert summation}\label{sec:dirichlet}

A Dirichlet series is a scalar series
\[
\sum_{n=0}^\infty \frac{a_n}{n^z}
\]
where $a_\bullet\in s(\mathbb{C})$ and $z$ is a complex variable. The best-known Dirichlet series is the Riemann zeta function
\[
\zeta(z) = \sum_{n=1}^\infty \frac{1}{n^z}.
\]
Another well-known Dirichlet series is one whose coefficients are given by $a_n = \mu(n)$, where
\[
\mu(n) = \begin{dcases*}
        1 & $n$ is square-free with an even number of prime factors,\\
        -1 & $n$ is square-free with an odd number of prime factors,\\
        0 & $n$ is not square-free,
    \end{dcases*}
\]
is the M\"obius function. It turns out that for any $z \in \mathbb{C}$ with $\re(z) > 1$,
\[
\sum_{n=1}^\infty \frac{\mu(n)}{n^z} = \frac{1}{\zeta(z)}
\]
and
\begin{equation}\label{eq:mob2}
    \lim_{z \to 1}\sum_{n=1}^\infty \frac{\mu(n)}{n^z} = 0.
\end{equation}
An important application of the scalar Lambert summation \cite[Lemma~2.3.7]{Lambert-book} is to show that
\begin{equation}\label{eq:mob}
    \sum_{n=1}^\infty \frac{\mu(n)}{n} \overeq{L} 0
\end{equation}
and our goal here is to verify a matrix analogue numerically.

It is straightforward to extend the definitions above. A matrix Dirichlet series is a matrix series
\[
\sum_{n=0}^\infty a_n n^{-X}
\]
where $X$ is a complex matrix variable that takes values in $\mathbb{C}^{d \times d}$ and 
\[
n^X \coloneqq \exp(\log(n)X),
\]
with $\exp$ the matrix exponential function \cite[Chapter~10]{funcofmat}. Our numerical experiments show that if $X \in \mathbb{C}^{d \times d}$ has $\re(X) \succeq I$,  then
\begin{equation}\label{eq:mob3}
    \sum_{n=1}^\infty \mu(n) n^{-X} 
\end{equation}
is Lambert summable in the sense of Definition~\ref{def:lambert} and 
\begin{equation}\label{eq:mob4}
    \lim_{X \to I }  \biggl( \sum_{n=1}^\infty \mu(n) n^{-X} \biggr) \overeq{L} 0.
\end{equation}
This is a matrix analogue of \eqref{eq:mob2} and \eqref{eq:mob}. Unlike the scalar version in \eqref{eq:mob2}, which is conventionally summable, our matrix version in \eqref{eq:mob4} requires Lambert summation as the matrix Dirichlet series \eqref{eq:mob3} is not conventionally summable if $1 \in \lambda(X)$, but is nevertheless Lambert summable.

\begin{figure}[htb]
\includegraphics[trim={10em 51ex 10em 55ex},clip,width=0.55\textwidth]{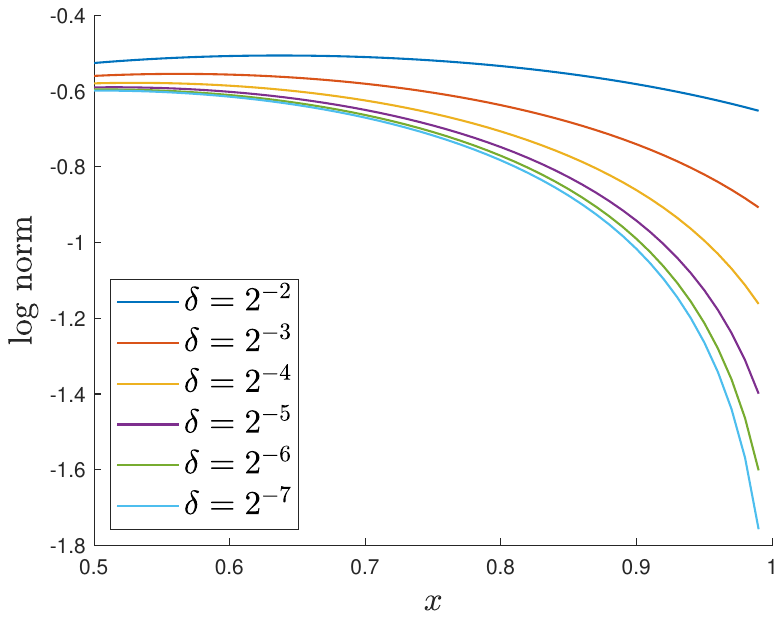}
    \caption{Lambert approximation of the Dirichlet series.}
    \label{fig:mobius}
\end{figure}

To verify \eqref{eq:mob4} numerically, we generate random matrices $X \in \mathbb{C}^{1000 \times 1000}$ with $\re(X) \succeq I$ and $\lVert X-I\rVert = \delta$ for $\delta = 2^{-2},2^{-3},\dots,2^{-7}$, and compute 
\[
\widehat{S} = (1-x)\sum_{n=1}^{10000} \frac{nx^n}{1-x^n} \mu(n) n^{-X}
\]
to approximate the Lambert sum as $x\to 1^-$. As shown in Figure~\ref{fig:mobius}, for each $\delta$, $\lVert \widehat{S} \rVert$ approaches a limiting value as $x \to 1^-$, and $\lVert \widehat{S} \rVert \to 0$ as $\delta \to 0$ or, equivalently, $X \to I$.

\subsection{Recursive versus compensated summations}\label{sec:comp_numerical}

We present two sets of experiments to compare recursive summation in \eqref{eq:naive} with compensated summation in Algorithm~\ref{alg:comp}, focusing on how the errors scale with respect to series length and matrix dimensions. 

For $X \in \mathbb{R}^{d\times d}$, we consider the $n$-term Neumann series
\begin{equation}\label{eq:neu}
    \sum_{k=0}^{n-1} X^k = (I-X^n)(I-X)^{-1} \eqqcolon S
\end{equation}
for  fixed $n = 1000$ and $d = 1,\dots, 1000$. We also consider its $n$-term Hadamard analogue, i.e., with power taken with respect to the Hadamard product
\begin{equation}\label{eq:had}
    \sum_{k=0}^{n-1} X^{\circ k}  = S_\circ, \quad s_{ij}^\circ = \frac{1-x_{ij}^n}{1-x_{ij}}, \quad i,j=1,\dots,d,
\end{equation}
for fixed $d = 1000$ and $n = 1,\dots, 5000$. In both cases we have the respective closed-form expressions for $S$ and $S_\circ$ on the right of \eqref{eq:neu} and \eqref{eq:had} that give their exact values and thereby permit calculation of forward errors.

We compute $\widehat{S}$, the sum on the left of \eqref{eq:neu}, and  $\widehat{S}_\circ$  the sum on the left of \eqref{eq:had} using both recursive summation in \eqref{eq:naive} and compensated summation in Algorithm~\ref{alg:comp}. The forward errors $\norm{\widehat{S} - S}_{\method{F}}$ and $\norm{\widehat{S}_\circ - S_\circ}_{\method{F}}$ are shown in Figures~\ref{fig:d} and \ref{fig:n} respectively. The result is clear: compensated summation is consistently more accurate than recursive summation, particularly with respect to increasing series length $n$, where the increase in errors follow significantly different trends.

While our forward error bound \eqref{eq:compensated_bound} predicts that the errors in compensated summation should be free of any dependence on $n$, this is assuming that we know the $k$th term \emph{exactly}. In our sum \eqref{eq:had}, the $k$th term $X^{\circ k}$ is \emph{computed}, and the increase in errors we see in Figure~\ref{fig:n} is a result of the multiplication errors accumulating in $\widehat{S}_\circ $ as $n$ increases.

\begin{figure}[htb]
  \begin{subfigure}{0.49\textwidth}
    \includegraphics[trim={9.5em 51ex 10.5em 53ex},clip,width=\textwidth]{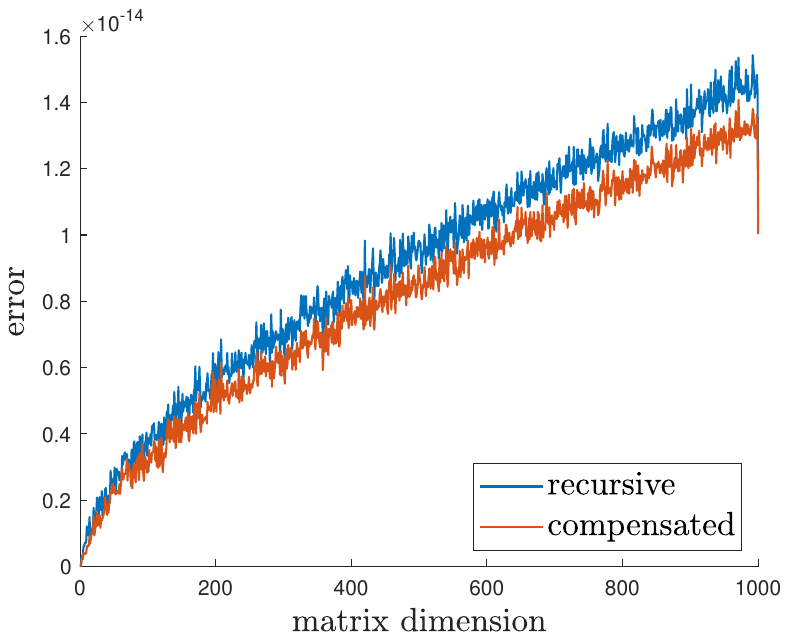}
    \caption{increasing matrix dimension $d$} \label{fig:d}
  \end{subfigure}%
  \begin{subfigure}{0.49\textwidth}
    \includegraphics[trim={9.5em 51ex 10.5em 53ex},clip,width=\textwidth]{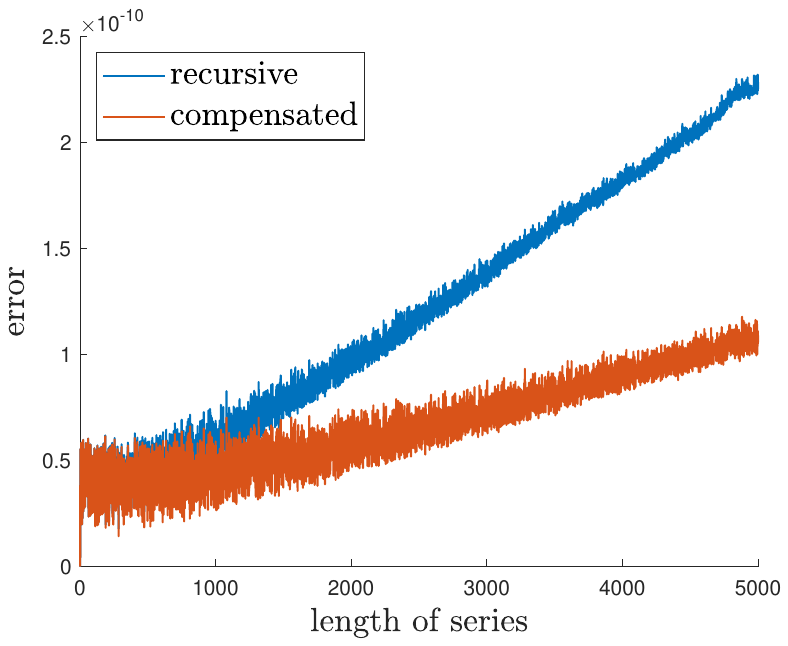}
    \caption{increasing series length $n$} \label{fig:n}
  \end{subfigure}%
\caption{Errors for recursive and compensated summation algorithms.} \label{fig:recur_comp}
\end{figure}

In case the reader is wondering why we did two different sets of experiments with respect to standard and Hadamard products: Hadamard products will not  reveal the dependence on $d$ in Figure~\ref{fig:d} as they are computed entrywise; whereas standard products will result in the multiplicative errors masking the trend in Figure~\ref{fig:n} showing dependence on $n$, as computing $X^k$ requires an order of magnitude more multiplications than computing $X^{\circ k}$. 

\section{Conclusion}\label{sec:conclusion}

This article is likely the first systematic study of summation techniques, both theoretical and numerical, for matrix series. Indeed we are unable to find much discussion of general numerical algorithms even for summing \emph{conventionally convergent} matrix series, let alone the more convoluted summation methods for matrix series divergent in the conventional sense. The handful of previous works we found \cite{Higham-Schur--Parlett,funcofmat,scaling-squaring} had all focused on conventional summation of specific matrix Taylor series related to matrix functions, and said nothing of other summation methods or more general matrix series. Despite the length of our article, it still leaves significant room for future work, with several immediate open problems that we will briefly describe.

Our extensions of matrix Abelian mean in Definition~\ref{def:Abemeans}, matrix Lambert sum in Definition~\ref{def:lambert}, weak and strong matrix Borel sums in Definitions~\ref{def:wborel} and \ref{def:sborel}, leave open the question of whether one may further extend them by replacing the scalar parameter $x$ in these definitions by a positive definite matrix. One may also ask a similar question of the matrix Mittag-Leffler sum in Definition~\ref{def:ml}: Could the gamma function be replaced by the matrix gamma function \cite{gamma}?

Another aspect beyond the scope of this article is that of conditioning, which likely explains the surprising accuracy of Euler method over matrix inversion uncovered in Section~\ref{sec:neu}. Note that the left- and right-hand sides of  \eqref{eq:neumann}, despite being equal in value, involve two different computational problems and almost surely have entirely different condition numbers. What is lacking is a study of the condition numbers of the summation methods in Sections~\ref{sec:seq} and \ref{sec:func}.

The numerical methods in Sections~\ref{sec:gen} and \ref{sec:pow} are mainly designed with accuracy in mind. They work well when adapted for matrix series and in conjunction with the summation methods in Sections~\ref{sec:seq} and \ref{sec:func}. When it comes to speed, there are many acceleration methods for scalar series such as Aitken's $\delta^2$-process and the vector $\varepsilon$-algorithm \cite{extrapolation, epsilon_algorithm, RecipesC}, but these involve nonlinear transforms and adapting them for matrix series is a challenge we save for the future.

As we alluded to in the introduction, these summation methods will allow for numerical investigations of ``random matrix series,'' one that has its $k$th term $A_k$ randomly generated according to some distributions like Wishart or GUE \cite{rand}. Many celebrated results in random matrix theory were indeed discovered first through numerical experiments and only rigorously proved much later.

\subsection*{Acknowledgments} This work is partially supported by the DARPA grant HR00112190040, the NSF grants DMS 1854831 and ECCS 2216912, and a Vannevar Bush Faculty Fellowship ONR N000142312863. We thank the two anonymous reviewers for their very helpful comments and suggestions.
    
\bibliographystyle{abbrv}

\end{document}